\documentclass{amsart}
\usepackage[utf8]{inputenc}

\input{defs.sty}

\title{Determinant map for the prestack of Tate objects}
\author{Aron Heleodoro}
\date{\today}

\dedicatory{In memory of Tom Nevins.}

\begin{document}

\maketitle


\begin{abstract}
    We construct a map from the prestack of Tate objects over a commutative ring $k$ to the stack of $\Gm$-gerbes. The result is obtained by combining the determinant map from the stack of perfect complexes as proposed by Sch\"urg-To\"en-Vezzosi with a relative $S_{\bullet}$-construction for Tate objects as studied by Braunling-Groechenig-Wolfson. Along the way we prove a result about the K-theory of vector bundles over a connective $\Einf$-ring spectrum which is possibly of independent interest.
\end{abstract}



\tableofcontents

\section*{Introduction}

Let $k$ be a field of characteristic $0$ and consider $\GL_{\infty}$ the group of continuous automorphisms of $k((t))$\footnote{Equivalently, automorphisms of $k((t))$ as a Tate object.}. It is a standard fact (see \cite{Faltings}, \cite{Tits} and \cite{FZ}*{\S 1}) that one has a canonical central extension
\begin{equation}
\label{eq:one-dimensional-central-extension}
1 \ra \Gm \ra \widehat{\GL}_{\infty} \ra \GL_{\infty} \ra 1.    
\end{equation}

In \cite{FZ} (see also \cite{Osipov-Zhu} for the case when $k$ is any commutative ring) the authors construct an analogue of this central extension for $\GL^{(2)}_{\infty}$, i.e.\ the group of automorphisms of $k((t_1))((t_2))$ in some appropriate sense (see \cite{FZ}*{\S 3.1.2} for a discussion of the topology used), equivalently as a $2$-Tate object over $k$. Namely, they construct a central extension of $\GL^{(2)}_{\infty}$ by $\BGm$, the Picard groupoid over a point.

This paper originated from two motivations: 
\begin{enumerate}[1)]
    \item trying to understand the analogue of \cite{FZ} for higher iterations, that is construct central extensions of the group of automorphisms of $k((t_1))\cdots((t_n))$;
    \item have a construction of the group central extension that also gives the Lie algebra central extension, namely a central extension of $\mathfrak{gl}^{(n)}_{\infty}$, the endomorphism Lie algebra of $k((t_1))\cdots((t_n))$.
\end{enumerate}

To accomplish the first goal, we need to work in a much more general context than \cite{FZ}. The first generalization is that our algebro-geometric objects are higher stacks, i.e.\ functors of points that take values in spaces rather than sets. This step is clearly necessary once one is convinced that the central extensions should be by higher deloopings of $\Gm$, i.e.\ $\K(n-1,\Gm)$. The second generalization is that we need to consider derived stacks, or rather prestacks. One reason is that it is not clear how to make sense of a geometric object classifying higher Tate objects over a scheme without considering perfect complexes. The second, more serious, reason is that to carry out our strategy to achieve the second goal, we need to make use of the theory of parametrized formal moduli problems, as developed in \cites{GR-I,GR-II} and, as explained in their introduction, derived geometry is essential for that.

The main result of this paper is the construction of a determinant map at the level of prestacks.

\begin{thmannounce}[Corollary \ref{cor:higher-determinant-for-prestacks}]
\label{thmannounce:determinant-map}
For $k$ a commutative ring, one has a map of prestacks over $k$
\[\sD: \sTate \ra \sBPicgr,\]
where $\sTate$ is the prestack of Tate objects and $\sBPicgr$ is the prestack classifing graded $\Gm$-gerbes.

More generally, for any $n \geq 1$ one has a map of prestacks
\[\sD^{(n)}: \sTate^{(n)} \ra \sBPicgrn{n}.\]
\end{thmannounce}

The proof of this result rests in two auxiliary theorems which can be of independent interest. They hold in different regimes, Theorem \ref{thmannounce:K-theory} holds for arbitrary connective $\Einf$-algebras, whereas Theorem \ref{thmannounce:Vect-classical} only holds for simplicial commutative rings over an arbitrary commutative ring $k$. 

The first result is of a K-theoretic nature. It is a variant on the theorem of the heart (cf. \cite{Barwick-heart}) for weight structures.

\begin{thmannounce}[Corollary \ref{cor:K-theory-Perf}]
\label{thmannounce:K-theory}
Let $R$ be a connective $\Einf$-ring. The canonical inclusion $\Vect(R) \ra \Perf(R)$, of the subcategory generated by retracts of finite direct sums of $R$ induces an equivalence
\[\K(\Vect(R)) \overset{\simeq}{\ra} \K(\Perf(R))\]
of connective $K$-theory spectra.
\end{thmannounce}

While this paper was being written Theorem \ref{thmannounce:K-theory} appeared as the main result of \cite{Fontes}. We notice that a closely analogous statement was already present in \cite{Sosnilo}*{Corollary 4.1} and the classical reference \cite{EKMM} has a form of this statement for categories of modules over an $\mbox{E}_1$-algebra.


The second result is a statement about when a prestack is smooth-extended, i.e.\ it is the left Kan extension of its restriction to the category of smooth affine schemes.

\begin{thmannounce}[Theorem \ref{thm:sVect-is-smooth-extended}]
\label{thmannounce:Vect-classical}
Let $k$ be a commutative ring. The derived stack $\sVect$ over $k$ is smooth-extended, i.e.\ one has an equivalence of stacks
\[
\LLKEsm(\smooth{\sVect}) \ra \sVect,
\]
between $\sVect$ and the left Kan extension of the restriction of $\sVect$ to smooth (classical) affine schemes over $k$ followed by sheafification.
\end{thmannounce}

The proof of Theorem \ref{thmannounce:Vect-classical} uses an observation by A.\ Mathew (see \cite{EHKSY}*{Appendix A}). One also needs a concrete description of $\sVect$ in terms of a disjoint union of $\BGL_n$ for arbitrary $n$.

One reason for the formulation in such a generality of the index map from Theorem \ref{thmannounce:determinant-map} is the ease with which we can use it to construct the sought-for central extensions.

\begin{thmannounce}[See \S \ref{subsec:central-extensions}]
\label{thmannounce:central-extensions}
For $k$ an arbitrary commutative ring. Let $\sLG$ denote the stack of automorphisms of $k((t))$ as a Tate object, one has a canonical central extension
\[1 \ra \Gm \ra \widetilde{\sLG} \ra \sLG \ra 1.\]

Furthermore, for any $n\geq 2$ one has a canonical central extension
\[1 \ra \sBPicgrn{n-2} \ra \widetilde{\sLGL{n}} \ra \sLGL{n} \ra 1,\]
where $\sLGL{n}$ is the stack of automorphisms of the $n$-Tate object $k((t_1))\cdots((t_n))$.
\end{thmannounce}

The central extensions of Theorem \ref{thmannounce:central-extensions} agree with the ones constructed in the literature where they have been defined. Namely, when restricted to classical affine schemes and truncated to a scheme (i.e.\ Set-valued functor of points) the extension $\sLG$ agrees with the central extension constructed by Frenkel-Zhu in \cite{FZ}.


One advantage of our formulation is that it provides a framework to construct the central extensions of the corresponding higher loop Lie algebras in relation with the central extension of groups. Namely, following the steps of \cite{GR-I} for prestacks of Tate type one can rigorously take the derivate of the maps in Theorem \ref{thmannounce:determinant-map} to obtain a map between dg Lie algebras over $k$, i.e.\ a central extension of the Lie algebra underlying $\sLG$ and $\sLGL{n}$ for $n\geq 2$. The central extensions obtained for those Lie algebras should also recover results from \cite{FHK}. This will be the subject of further work.

We briefly describe the contents of each section. Section 1 recalls the K-theory of Waldhausen categories; the expert can safely skip this material. Section 2 has a discussion of functors left Kan extended from smooth algebras and proves Theorem \ref{thmannounce:Vect-classical}. Section 3 develops the main input for the construction of the higher determinant map, that is the determinant for the prestack of perfect complexes. Section 4 constructs the determinant for Tate objects and the central extensions.

\subsection*{Conventions.}

Unless otherwise stated we have the following terminology:
\begin{itemize}
    \item this paper is written in the language of $\infty$-categories (as developed in \cite{HTT}), unless otherwise noted all categorical concepts (categories, functors, etc.) and constructions (limits, left Kan extensions, etc.) are to be understood in the sense of $\infty$-categories;
    \item an ordinary category shall mean a $1$-category in the usual sense;
    \item the term space means an $\infty$-groupoid, we will denote by $\Spc$ the ($\infty$-)category of spaces;
    \item all schemes and stacks are derived, a prestack means a functor from (derived) schemes to spaces (for most of the generality of this paper one can take the set up of \cite{SAG}*{Chapter 25}\footnote{See \S \ref{subsec:stacks-smooth-extended} below for more details on our conventions.});
    \item a classical scheme means a scheme in the ordinary sense, i.e.\ non-derived.
\end{itemize}

\subsection*{Acknowledgments}

It is my pleasure to thank Ben Antieau, Emily Cliff, Elden Elmanto and Jeremiah Heller for discussions about different contents of this paper, specially I would like to thank A.\ Mathew for discussions about how to prove certain results over arbitrary commutative rings. I am specially grateful to Nick Rozenblyum whose continued support during this project made it possible. Finally, I would like to thank the anonymous referee for comments that improved the presentation and the statement of some of the results of this paper.

\section{Preliminaries on K-theory}
\label{sec:K-theory}

This section recalls some definitions and results about the K-theory of Waldhausen categories\footnote{The $\infty$-categorical analogue of an (ordinary) category with cofibrations.}. This theory was originally developed in \cite{Barwick}, and mostly all the results, with the exception of Theorem \ref{thm:K-theory-heart-of-weight}, are contained in there. In the appendix we present some proofs which could be illuminating. The online notes \cite{Lurie-K-theory} also contain a good amount of what is discussed in this section. Note that to prove the equivalence (Corollary \ref{cor:K-theory-Perf}) we need to construction the determinant map in \S \ref{sec:determinant}, it is not sufficient to work with the K-theory of stable categories (cf. \cite{BGT}).

\subsection{Set up and definitions}

\subsubsection{Categories with cofibrations}
\label{subsubsec:Waldhausen-categories}

We introduce the main kind of categories of which we will study the K-theory.

\begin{defn}
\label{defn:Waldhausen-categories}
A \emph{category with cofibrations} is a pair $(\sC,\co(\sC))$, $\sC$ a presentable pointed category and $\co(\sC)$ a class of morphisms satisfying
\begin{enumerate}[(i)]
    \item for any $X \in \sC$, the map $\ast \ra X$ is in $\co(\sC)$;
    \item the composition of two cofibrations is a cofibration;
    \item for a map $f:X \ra Y$ in $\co(\sC)$ and an arbitrary map $X \ra X'$ consider the pushout diagram
    \begin{center}
        \begin{tikzcd}
        X \ar[r,"f"] \ar[d] & Y \ar[d] \\
        X' \ar[r,"f'"] & Y'
        \end{tikzcd}
    \end{center}
    the map $f':X' \ra Y'$ is also a cofibration.
\end{enumerate}
\end{defn}

\begin{rem}
We will also use the name \emph{Waldhausen category} for a category with cofibrations, as introduced in \cite{Barwick}*{Definition 2.7}. One can realize the category of Waldhausen categories denoted by $\Wald$ as a subcategory of an appropriate category of pairs (see \cite{Barwick}*{\S 2.13 and 2.14} for details). 

If the cofibrations are clear from the context we will drop $\co(\sC)$ from the notation and just write $\sC$ for a Waldhausen category.
\end{rem}

\begin{example}
\label{ex:minimal-cofibrations}
For $\sC$ a stable category, consider $\co_{\rm min}(\sC)$ the smallest class of morphisms containing $0 \ra X$ for all objects $X$ in $\sC$ and stable under pushouts and compositions. This defines a category with cofibrations $(\sC,\co_{\rm min}(\sC))$.
\end{example}

\begin{example}
For $\sC$ a category with finite direct sums\footnote{I.e.\ $\sC$ admits finite limits and colimits and they agree up to a contractible space of choices.}, let $\co_{\rm split}(\sC)$ be the class of maps $f:X \ra Y$ such that $f$ is isomorphic to a map $X \overset{\id_X \oplus 0}{\ra} X \oplus Z$ for some $Z \in \sC$. This gives $(\sC,\co_{\rm split}(\sC))$ the structure of a category with cofibrations, we refer to this choice as the \emph{split cofibrations}.
\end{example}

\begin{example}
\label{ex:maximal-cofibrations}
For any presentable pointed category $\sC$ we can take $\co_{\rm max}(\sC)$ to consist of all morphisms in $\sC$.
\end{example}

\begin{example}
Let $\b{C}$ be a pointed ordinary category with a class of cofibrations $\co(\b{C})$ (in the sense of Waldhausen \cite{Waldhausen}*{\S 1.1}). Then $(\N\b{C},\N\co(\b{C}))$ is a Waldhausen category.
\end{example}

We now introduce Waldhausen's $S_{\bullet}$-construction for Waldhausen categories.

\begin{defn}
For every $n \geq 0$ we let $\sS_{n}\sC$ be the subcategory of $\Fun(N([n]\overset{\leq }{\times} [n]),\sC)$\footnote{Here $[n]\overset{\leq }{\times} [n]$ is the subset of pairs $(i,j) \in [n]\times [n]$ such that $i \leq j$ and morphisms are those restricted from the category $[n] \times [n]$.} satisfying:
\begin{enumerate}[(i)]
    \item for all $i \leq j \leq k$ the map $X_{i,j} \ra X_{i,k}$ is a cofibration;
    \item for all $i$, $X_{i,i} \simeq \ast$;
    \item for all $i \leq j \leq k$ the diagram
    \begin{center}
        \begin{tikzcd}
        X_{i,j} \ar[r] \ar[d] & X_{i,k} \ar[d] \\
        \ast \ar[r] & X_{j,k}
        \end{tikzcd}
    \end{center}
    is a pushout.
\end{enumerate}
\end{defn}

It is clear that the assignment $[n] \mapsto \sS_{n}\sC$ determines a simplicial object in the category of pointed presentable categories. We will denote by $S_{\bullet}\sC$ the simplicial object in $\Spc$ obtained by passing to the underlying $\infty$-groupoid levelwise.

\begin{rem}
Actually, for each $[n] \in \Delta^{\rm op}$ the category $\sS_n(\sC)$ comes endowed with cofibrations, and one can consider $\sS_{\bullet}\sC$ as a simplicial object in $\Wald$ (see \cite{Barwick}*{Definition 5.6.}).
\end{rem}

One single iteration of the $S_{\bullet}$-construction is enough to define the underlying K-theory space.

\begin{defn}
\label{defn:K-theory-space}
For $(\sC,\co(\sC))$ a Waldhausen category, its K-theory space is
\[K_{\rm Spc}(\sC) = \Omega\left|S_{\bullet}\sC\right|.\]
\end{defn}

To upgrade K-theory to a spectrum\footnote{We will not consider the non-connective version of K-theory in this article, so any time we refer to the K-theory spectrum it is tacitly assumed to be its connective version.} we need to consider iterations of the $S_{\bullet}$-construction. We first introduce some notation on multi-simplicial objects to be able to do that.

For $n \in \Delta^{\rm op}$ let $\Ar_n = N([n]\overset{\leq}{\times} [n])$, and for each $k\geq 1$ we will denote by $[n^{(k)}] = ([n_1],\ldots,[n_k])$ an element of $(\Delta^{\rm op})^k$. Similarly, let $\Ar_{n^{(k)}} = \Ar_{n_1}\times \cdots \times \Ar_{n_k}$.

\begin{defn}
\label{defn:iterated-S-construction}
Let $k \geq 1$. 
For each object $n^{(k)} \in (\Delta^{\rm op})^k$ we define the subcategory $\sS_{n^{(k)}}\sC$ to be the fiber product
\begin{center}
    \begin{tikzcd}
    \sS_{n^{(k)}}\sC \ar[r] \ar[d] & \Fun(\Ar_{n^{(k)}},\sC) \ar[d] \\
    \prod^k_{i=1}\sS_{n_i}\sC \ar[r] & \prod^k_{i=1}\Fun(\Ar_{n_i},\sC).
    \end{tikzcd}
\end{center}
where the vertical map is the restriction to each factor and the horizontal one is the canonical inclusion.

The $k$th iterated $S_{\bullet}$-construction on $\sC$ is the functor
\[S_{\bullet_{k}}\sC = \sS_{\bullet_k}\sC^{\simeq},\]
where we pass to underlying $\infty$-groupoids levelwise.
\end{defn}


Let $\Spc^{(n)} = \Fun\left(\b{N}(\Delta^{\times n})^{\rm op},\Spc\right)$ denote the category of $n$-simplicial spaces, by adjunction this is equivalent to $\Fun\left(\b{N}(\Delta^{\rm op}),\Spc^{(n-1)}\right)$, i.e.\ simplicial objects in $(n-1)$-simplicial spaces. We consider the adjunction\footnote{This is a partial skeleton and relative truncation adjunction.}
\[
\begin{tikzcd}
\Fun\left(\b{N}(\Delta^{\rm op}),\Spc^{(n-1)}\right) \ar[r,shift left,"\tr^{(1)}_{1}"] & \Fun\left(\b{N}(\Delta^{\rm op}_{\leq 1}),\Spc^{(n-1)}\right) \ar[l,shift left,"\sk^{(1)}_{1}"]
\end{tikzcd}
\]
where $\Delta^{\rm op}_{\leq 1} \subset \Delta^{\rm op}$ is the full subcategory generated by $[0]$ and $[1]$.

Hence, for any $X_{\bullet_{n}} \in \Spc^{(n)}$ one has a counit map
\[(\sk^{(1)}_{1}\circ \tr^{(1)}_{1}X)_{\bullet_n} \ra X_{\bullet_n}\]
which induces a map upon geometric realization\footnote{I.e. the total geometric realization of the $n$-simplicial object.}
\begin{equation}
\label{eq:geometric-realization-1-skeleton-inclusion}
\left|\sk^{(1)}_{1}\circ \tr^{(1)}_{1}X_{\bullet_n}\right| \ra \left|X_{\bullet_n}\right|.    
\end{equation}

\begin{lem}
\label{lem:suspension-iterated-S-construction}
For every $k \geq 1$ there exists a canonical map
\begin{equation}
\label{eq:suspension-iterated-S-construction}
    \Sigma|S_{\bullet_k}\sC| \ra |S_{\bullet_{k+1}}\sC|.
\end{equation}
\end{lem}

\begin{proof}
For $S_{\bullet_k}\sC$ we notice that 
\[\tr^{(1)}_{1}S_{\bullet_{k}}\sC([0],[n^{(k-1)}]) \simeq \ast, \;\;\mbox{and}\;\;\tr^{(1)}_{1}S_{\bullet_{k}}\sC([1],[n^{(k-1)}]) \simeq S_{\bullet_{k-1}}\sC([n^{(k-1)}]) \]
for all $[n^{(k-1)}] \in (\Delta^{\rm op})^{n-1}$. 

Thus one has
\[\left|\sk^{(1)}_{1}\circ \tr^{(1)}_{1}S_{\bullet_k}\sC\right| \simeq \left|
\begin{tikzcd}
S_{\bullet_{k-1}}\sC \ar[r,shift left] \ar[r,shift right] & \ast
\end{tikzcd}
\right| \simeq \Sigma\left|S_{\bullet_{k-1}}\sC\right|.
\]

So the map (\ref{eq:geometric-realization-1-skeleton-inclusion}) applied to $S_{\bullet_k}\sC$ yields the desired map.
\end{proof}

Recall that one concrete way to think about the category of spectra $\Spctr$ is as
\[\Spctr = \colim_{\Sigma^{\circ n}}\Spc,\]
i.e.\ the colimit of $\Spc$ in the category of presentable categories and left exact functors under iterations of the suspension functor.

\begin{defn}
\label{defn:K-theory-spectrum}
For any category with cofibrations $(\sC,\co(\sC))$ we let
\[\K(\sC) = (|S_{\bullet}\sC|,|S_{\bullet_2}\sC|,\ldots)\]
be its K-theory spectrum, with structure maps given by Lemma \ref{lem:suspension-iterated-S-construction}.
\end{defn}

\begin{rem}
\label{rem:functoriality-S-construction}
For $F: \sC \ra \sD$ a functor preserving pushouts and the distinguished point of $\sC$, $F$ applied to $\sS_n\sC$ levelwise factors through $\sS_{n}\sD$ inside $\Fun(\Ar_n,\sD)$. Thus, any such functor induces a map between the $S_{\bullet}$-constructions and its iterations, hence between $\K(\sC) \ra \K(\sD)$.
\end{rem}

\begin{rem}
\label{rem:relation-K-theory-spectrum-to-K-theory-space}
Notice that given the K-theory spectrum associated to a Waldhausen category $\sC$, we can recover its K-theory space by
\[\Omega(\K(\sC)_0)\]
where $\K(\sC)_0$ denotes the $0$th space underlying the spectrum $\K(\sC)$\footnote{This is sometimes denoted $\Omega^{\infty}\K(\sC)$ in the literature.}.
\end{rem}

\begin{rem}
One has a canonical map from the underlying $\infty$-groupoid of a Waldhausen category to its K-theory space
\[
\imath: \sC^{\simeq} \ra \K_{\rm Spc}(\sC).
\]
Indeed, since $\left|\sk\circ\tr\sS_{\bullet}\sC\right| \simeq \sigma \sC^{\simeq}$, (\ref{eq:geometric-realization-1-skeleton-inclusion}) applied to $S_{\bullet}\sC$ gives $\Sigma \sC^{\simeq} \ra \left|\sS_{\bullet}\sC\right|$, the map $\imath$ is obtained by taking loops.

In fact, the adjunction
\[
\begin{tikzcd}
\Spc \ar[r,shift left,"\Sp(-)"] & \Spctr \ar[l,shift left,"(-)_0"]
\end{tikzcd}
\]
between the category of spaces and the category of spectra gives a map
\begin{equation}
\label{eq:imath-for-spectra}
\Sp(\imath): \Sp(\sC^{\simeq}) \ra \Sp(\K(\sC)_0) \simeq \Sp(\K_{\rm Spc}(\sC)) \ra \K(\sC).    
\end{equation}
We will abuse notation and denote the map (\ref{eq:imath-for-spectra}) by $\imath$ and not make a distinction in notation when we consider $\sC^{\simeq}$ as a spectrum or a space.

\end{rem}


\begin{rem}
\label{rem:S-construction-is-delooping}
In fact, the map in Lemma \ref{lem:suspension-iterated-S-construction} is an equivalence. Indeed, this follows from \cite{Barwick}*{Proposition 6.10}. As \cite{Barwick} we let $\D(\Wald)$ denote the category obtained by formally adjoining sifted colimits to $\Wald$. Thus, given any Waldhausen category $\sC$, the colimit $|\sS_{\bullet}\sC|$ belongs to $\D(\Wald)$ and $\left|\sS_{\bullet_2}\sC\right|$ as well. Finally, we claim that the functor that sends $\sC \in \Wald$ to $|S_{\bullet}\sC| \in \Spc$ factors through $\D_{\rm fiss}(\Wald)$, from which our claim follows.
\end{rem}

\subsubsection{Additive K-theory}

In the case where $\sC$ has finite direct sums one can define its additive K-theory. This is easier to define than the K-theory spectrum via the $S_{\bullet}$-construction, but it is just as good in some situations, namely when all the cofibrations are split (see Theorem \ref{thm:additive-K-theory-agrees-when-split-cofibrations}). We don't know of a published reference for this, we learned of such results from \cite{Lurie-K-theory}.

Consider $(\sC,\oplus)$ a pointed symmetric monoidal stable category\footnote{I.e. this is a coCartesian symmetric monoidal structure (see \cite{HA}*{\S 2.4.3}).}, its underlying $\infty$-groupoid $\sC^{\simeq}$ canonically has the structure of an $\Einf$-monoid in spaces, namely the straightening of the coCartesian fibration defining the symmetric monoidal structure on $(\sC^{\simeq})^{\oplus}$. Thus, all one needs to produce a connective spectrum is to ensure that $\sC^{\simeq}$ is grouplike. A systematic way of achieving this is by performing a group completion, which can be concretely described as
\[\sC^{\simeq, \rm gp} \simeq \Omega\left|\B_{\bullet}\sC^{\simeq}\right|,\]
where $\B_{\bullet}\sC^{\simeq}$ is the simplicial object in spaces obtained from taking $n$ copies of $\sC^{\simeq}$ on each degree with the corresponding simplicial maps given by multiplication and projections.

\begin{defn}
\label{defn:additive-K-theory}
For $\sC$ a pointed category with finite coproducts, we let 
\[
\Ka(\sC)
\]
denote the connective spectrum associated to the group-like $\Einf$-monoid in spaces $\Omega\left|\B_{\bullet}\sC^{\simeq}\right|$\footnote{The interested reader is refered to \cite{HA}*{\S 5.2.6} for a discussion about this equivalence which we won't make explicit.}.
\end{defn}

It is immediate to see that the construction of additive K-theory is also functorial in $\sC$.

\begin{rem}
\label{rem:map-additive-to-K-theory}
For any Waldhausen category $\sC$ with finite direct sums, we notice that for each $n \geq 1$, one has a functor
\begin{align*}
\centering
G_n: \B_n\sC & \ra S_n\sC  \\
(X_1,\ldots,X_{n}) & \mapsto (X_1 \ra X_1\oplus X_2 \ra \cdots \ra \oplus^n_{i=1}X_i), 
\end{align*}
where the maps on the right are cofibrations by (i) and (iii) of Definition \ref{defn:Waldhausen-categories}. 

In particular, the above map induces a map
\[
\Ka(\sC) \ra \K_{\rm Spc}(\sC)
\]
between K-theory spaces. By following the correspondence between grouplike $\Einf$-spaces and connective spectra one also has a map $\Ka(\sC) \ra \K(\sC)$ of connective spectra.
\end{rem}

One has the following comparison between additive and usual K-theories.

\begin{thm}
\label{thm:additive-K-theory-agrees-when-split-cofibrations}
Suppose that $\sC$ is a Waldhausen category with split cofibrations, and finite direct sums, whose homotopy category $\sh\sC$ is additive. Then the functor
\[\Ka(\sC) \overset{\simeq}{\ra} \K(\sC)\]
described in Remark \ref{rem:map-additive-to-K-theory} is an equivalence.
\end{thm}

We provide a proof of this result in \S \ref{subsec:additive-K-theory-appendix} of the appendix.

\subsection{Main results}

\subsubsection{Additivity Theorem}

The Additivity Theorem is a basic important result in K-theory, in the context of ($\infty$-)categories with cofibrations it was originally proved by Barwick \cite{Barwick}. For our convenience we review this statement below.

Assume $\sC$ is a Waldhausen category which admits finite colimits. Given $\Fun(\Delta^1,\sC)$ the category of morphisms, one has a functor
\begin{align}
\label{eq:Cone-functor}
    F: \Fun(\Delta^1,\sC) & \ra \sC\times\sC \\
    (X \overset{f}{\ra}Y) & \mapsto (X,\Cofib(f)) \nonumber
\end{align}

\begin{thm}[Additivity]
\label{thm:additivity}
The map on K-theory induced by the functor $F$ 
\[\K(\Fun(\Delta^1,\sC)) \ra \K(\sC)\times \K(\sC)\]
is an equivalence of K-theory spectra.
\end{thm}

\begin{proof}
The above is a consequence of \cite{Barwick}*{Corollary 7.12.1}. In \cite{Barwick} one has that for any pre-additive theory, for instance $(-)^{\simeq}: \Wald \ra \Spc$, the functor that sends a Waldhausen category to its underlying $\infty$-groupoid, its \emph{additivization}\footnote{We refer the reader to \cite{Barwick}*{Theorem 7.4.} where this condition is spelled out in seven equivalent forms.} is given by
\[\K_{\rm Spc}(\sC) = ((\sC)^{\simeq})^{\rm add.} \simeq \Omega\left|S_{\bullet}\sC\right|.\]
The result now follows from \cite{Barwick}*{Theorem 7.4.4}.
\end{proof}

\begin{rem}
\label{rem:proof-of-additivity}
For completeness we give an independent direct proof of this result in \S \ref{subsec:additivity-appendix}, see Proposition \ref{prop:additivity-appendix}.
\end{rem}


\subsubsection{Fibration Theorem}

In this section we follow \cite{Barwick}*{\S 9}. The formulation of the analogue of Waldhausen's (generic) fibration theorem (cf. \cite{Waldhausen}*{\S 1.6.}) in the context of categories needs an $\infty$-categorical analogue of the notion of a category with weak equivalences for Waldhausen categories. Since weak equivalences already have a precise meanings associated to it, we will follow Barwick and use labelled morphisms to refer to the class of morphisms which are the $\infty$-categorical analogue of weak equivalences.
 
\begin{defn}
\label{defn:labelled-Waldhausen-category}
A \emph{labeled Waldhausen category} $(\sC,w\sC)$ is a Waldhausen category $\sC$ and a collection of \emph{labeled morphisms} $w\sC$, which 
\begin{enumerate}[(i)]
    \item contain all isomorphisms of $\sC$;
    \item satisfy the glueing axiom, that is for all $X_{00} \overset{\ell}{\ra} X_{01}$, $X_{10} \overset{\ell}{\ra} X_{11}$ and $X_{20} \overset{\ell}{\ra} X_{21}$ labeled morphisms, and maps $X_{20} \la X_{00} \hra X_{10}$ and $X_{21} \la X_{01} \hra X_{11}$ the induced map
\begin{equation}
    \label{eq:glueing-condition}
    X_{20}\sqcup_{X_{00}}X_{10} \ra X_{21}\sqcup_{X_{01}}X_{11}
\end{equation}
is labeled.
\end{enumerate}
\end{defn}

\begin{example}
Let $(\b{C},\co \b{C},w\b{C})$ be a category with cofibrations and weak equivalence, in the sense of Waldhausen (cf. \cite{Waldhausen}*{\S 1.2}), then $(\N \b{C}, \N\mbox{co}\b{C},\N w\b{C})$ is a labeled Waldhausen category (see \cite{Barwick}*{Example 9.3}).
\end{example}

\begin{example}
As in Examples \ref{ex:minimal-cofibrations} and \ref{ex:maximal-cofibrations} before, one can take $w\sC$ to contain only the isomorphisms of $\sC$, or all the morphisms of $\sC$, respectively.
\end{example}


Before stating Barwick's version of the generic fibration theorem we need to introduce some notation and a technical condition. 

Suppose $(\sC,w\sC)$ is a labeled Waldhausen category. We consider
\[\sB(\sC,w\sC) = \colim_{\Delta^{\rm op}}\sB_{n}(\sC,w\sC),\]
where 
\[\sB_{n}(\sC,w\sC) = \{X_0 \overset{\ell}{\ra} X_1 \overset{\ell}{\ra} \cdots \overset{\ell}{\ra} X_n \; | \; X_{i} \in \sC, \; \forall 0 \leq i \leq n\},\]
the superscript $\ell$ denotes that the morphism is labeled.

\begin{rem}
\label{rem:Wald-does-not-have-sifted-colimits}
Notice that $\sB(\sC,w\sC)$ does not belong to the category $\Wald$, since this category doesn't admit geometric realizations. However, it belongs to the category $\sP^{N\Delta^{\rm op}}_{\empty}(\Wald)$\footnote{See \cite{HTT}*{\S 5.6} for this notation, or also \cite{Barwick}*{\S 4.14}, where this category is denoted $D(\Wald)$.} obtained from $\Wald$ by formally adjoining geometric realizations. We will denote by $\jmath: \Wald \hra \sP^{N\Delta^{\rm op}}_{\empty}(\Wald)$ the natural inclusion.
\end{rem}

The technical condition for the fibration theorem to hold is an $\infty$-categorical analogue of Waldhausen's existence of cylinder functors. 


\begin{defn}
\label{defn:enough-cofibrations}
Let $\sC$ be a labeled Waldhausen category, suppose that there exists a functor 
\[F: \Fun(\Delta^1,\sC) \ra \Fun(\Delta^1,\sC)\]
and a natural transformation $\eta: \id \Rightarrow F$, such that:
\begin{enumerate}[(i)]
    \item $F$ takes labeled morphisms to labeled cofibrations;
    \item if $f$ is a labeled cofibration, then $\eta_f$ is an equivalence;
    \item if $f$ is labeled, then $\eta_f$ is objectwise labeled.
\end{enumerate}
Then we say that $(\sC,w\sC)$ has \emph{enough cofibrations}.
\end{defn}

\begin{rem}
In \cite{Barwick} Barwick defines the notion of $\sC$ having enough cofibrations as a property of the functor $\Hom(-,\sC): \Wald \ra \Spc$ (see \cite{Barwick}*{Definition 9.21}). He then proves (\cite{Barwick}*{Lemma 9.22}) that the concrete conditions of Definition \ref{defn:enough-cofibrations}. are sufficient to guarantee that $\sC$ has enough cofibrations. The reason we take those conditions as the definition is that they are what one can actually check about a given labeled Waldhausen category.
\end{rem}

The following is \cite{Barwick}*{Theorem 9.24} applied to the case of K-theory\footnote{Namely, the additive theory (cf. \cite{Barwick}*{Definition 7.1. for $\sE = \Spc$.}) $\phi$ is taken to be the K-theory space functor $\K$, i.e.\ the additivization (cf. \cite{Barwick}*{Definition 7.10.}) of $(-)^{\simeq}:\Wald \ra \Spc$.}.

\begin{thm}
\label{thm:fibration}
For $(\sC,w\sC)$ a labeled Waldhausen category with enough cofibrations, and let 
\[
\sC^{w} = \{X \in \sC \; | \; 0 \overset{\ell}{\ra} X \}.
\] 

The following is a pullback square in spectra
\begin{equation}
\label{eq:fibration-thm-square}
    \begin{tikzcd}
    \K(\sC^w) \ar[r] \ar[d] & \K(\sC) \ar[d] \\
    \ast \ar[r] & \jmath_{!}\K(\sB(\sC,w\sC))
    \end{tikzcd}
\end{equation}
where $\jmath_{!}\K$ denotes the left Kan extension of $\K: \Wald \ra \Spctr^{\geq 0}$\footnote{The reason that one can consider $K$ to be valued in connective spectra is \cite{Barwick}*{Proposition 7.10}.} to $\sP^{N\Delta^{\rm op}}_{\empty}(\Wald)$.
\end{thm}


\subsubsection{Cell Decomposition Theorem}

In this section we discuss a result analogous to \cite{Waldhausen}*{\S 1.7} for the case of Waldhausen categories. This section has some overlap with \cite{Fontes}. First we need a definition.

\begin{defn}
Let $\sC$ be a pointed category with finite colimits\footnote{In \cite{Fontes} stable $\infty$-categories are considered, however all that one needs to define a weight structure is a suspension functor on $\sC$.}. A \emph{weight structure} on $\sC$ is the data of a pair of subcategories $(\sC_{w \leq 0},\sC_{w\geq 0})$ satisfying:
\begin{enumerate}[(i)]
    \item $\sC_{w\geq 1} = \Sigma\sC_{w\geq 0} \subset \sC_{w\geq 0}$ and $\sC_{w \leq 0} \subset \sC_{w\leq 1} = \Sigma \sC_{w \leq 0}$;
    \item for $X \in \sC_{w \leq 0}$ one has
    \[
    \Hom_{\sC}(X,Y) \simeq 0
    \]
    for all $Y \in \sC_{w\geq 1}$;
    \item for any $X \in \sC$ there is a decomposition
    \[
    X' \ra X \ra X''
    \]
    where $X' \in \sC_{w \leq 0}$ and $X'' \in \sC_{w\geq 1}$.
\end{enumerate}
We will refer to the subcategory $\sC^{\heartsuit}_{w} = \sC_{w \leq 0} \cap \sC_{w \geq 0}$ as the heart of the weight structure.
\end{defn}

\begin{rem}
The notion of a weight structure was introduced in \cite{Bondarko} for a triangulated category, but it makes sense in any category with a suspension functor. In particular, if a stable category $\sC$ has a weight structure, then the triangulated category $\h\sC$ has a weight structure in the sense of \cite{Bondarko}*{Definition 4.1}.
\end{rem}

\begin{rem}
We notice that weight structure has a flavor very similar to that of a t-structure with an important distinction, as pointed out in \cite{Bondarko-summary}*{Remark 4.4. 4.}, the heart of a weight structure is a notion of projective objects in $\sC$. So it is \emph{not} the case that for a stable category $\sC$ with a weight structure the subcategory $\sC^{\heartsuit}_{w}$ is abelian.
\end{rem}

We recall another result in K-theory in this case simply for motivation.

\begin{thm}[Theorem of the heart (\cite{Barwick-heart})]
\label{thm:cell-decomposition}
For $\sE$ a stable category with a bounded t-structure, one has equivalences of K-theory
\[
\K(\sE^{\heartsuit}) \overset{\simeq}{\ra} \K(\sE_{\geq 0}) \;\;\; \mbox{and}\;\;\; \K(\sE_{\geq 0}) \overset{\simeq}{\ra} \K(\sE).
\]
\end{thm}

One would like to prove a version of the above result for weight structures. It is not a surprise that we need some finiteness conditions on the weight structure for this to hold.

\begin{defn}
\label{defn:non-degenerate-bounded-weight}
Given a weight structure $(\sC_{w\leq 0},\sC_{w \geq 0})$ one says that
\begin{enumerate}[(i)]
    \item the weight structure is \emph{non-degenerate} if
    \[
    \bigcap_{n \ra \infty}\sC_{w \geq n} = \bigcap_{n \ra -\infty} \sC_{w \leq n} = 0;
    \]
    \item the weight structure is \emph{bounded} if
    \[
    \bigcup_{n \geq 0} (\sC_{w \leq n} \cap \sC_{w \geq -n}) = \sC.
    \]
\end{enumerate}
\end{defn}

\begin{example}
For $\Spctr^{\omega}$ the stable category of finite spectra, the sphere spectrum $\bS$ generates a weight structure. As follows from the proof of Theorem 4.3.2.III (ii) in \cite{Bondarko}) this can be described more explicitly as follows. $\Spctr^{\omega}_{w \geq 0}$ is the smallest subcategory of $\Spctr^{\omega}$ containing $\Sigma^n\bS$ for $n\geq 0$ and closed under finite extensions and taking idempotents. In other words, $\Spctr^{\omega}_{w \geq 0}$ is the category of finite connective spectra, and $\Spctr^{\omega}_{w \leq 0}$ is defined to be the orthogonal complement of $\Spctr^{\omega}_{w \geq 1}$\footnote{Notice $\Spctr^{\omega}_{w \geq 1}$ is equivalent to the category of finite spectra whose non-positive homotopy groups vanish.}, i.e.\
the subcategory of $Y \in \Spctr^{\omega}$ such that
\begin{equation}
\label{eq:orthogonal-condition}
    \Hom(Y,X) \simeq 0
\end{equation}
for all $X \in \Spctr^{\omega}_{w \geq 1}$.

Since $\Spctr^{\omega}_{w \geq 1}$ is generated by $\Sigma^1\bS$ under finite extensions and idempotent completion, condition (\ref{eq:orthogonal-condition}) concretely says that $Y \in \Spctr^{\omega}_{w \leq 0}$ if and only if
\[
\Hom(Y,\Sigma^n\bS) \simeq 0 
\]
for all $n\geq 1$. That is $Y$ has integral homology concentrated in non-positive degree. 
\end{example}

\begin{thm}[\cite{Fontes}*{Theorem 4.1.}]
\label{thm:K-theory-heart-of-weight}
Suppose $\sC$ is a Waldhausen category with a non-degenerate and bounded weight structure, then
\[
\K(\sC^{\heartsuit}_{w}) \overset{\simeq}{\ra} \K(\sC) \;\;\; \mbox{and} \;\;\; \K(\sC^{\heartsuit}_{w}) \ra \K(\sC_{w \geq 0})
\]
are equivalences of spectra.
\end{thm}

We give a proof of this theorem in \S \ref{subsec:weight-appendix} of the appendix. 

\subsection{Application: K-theory of vector bundles}

In this section we apply Theorem \ref{thm:K-theory-heart-of-weight} to the case of vector bundles over an affine scheme. Let $R$ be a connective $\Einf$-ring spectrum, and consider the category $\Perf(R)$ of compact objects in $\Mod_R$. Recall that the category $\Vect(R)$ of finitely generated and projective $R$-modules is the smallest subcategory of $\Mod_R$ containing $R$ and closed under finite direct sums and retracts.

We denote by $\Mod^{\geq n}_R$ the subcategory of $n$-connective $R$-modules, i.e. those whose homotopy groups vanishes for all $k < n$. One has the following concept introduced in \cite{Lurie-K-theory}*{Lecture 19}.

\begin{defn}
\label{defn:projective-amplitude}
An $R$-module $M$ is said to have \emph{projective amplitude} $\leq n$, if for any $N\in \Mod^{\geq (n+1)}_R$ one has
\[
\Hom_R(M,N) \simeq 0.
\]
\end{defn}

The next result is crucial to endow the category $\Perf(R)$ with a weight structure.

\begin{lem}
\label{lem:heart-is-projective}
For $M\in \Perf(R)$ the following are equivalent:
\begin{enumerate}[(1)]
    \item $M$ is finitely generated and projective;
    \item $M$ is connective and has projective amplitude $\leq 0$.
\end{enumerate}
\end{lem}

\begin{proof}
(1) $\Rightarrow$ (2): 
Recall that by definition an $R$-module $M$ is projective\footnote{See \cite{HA}*{Definition 7.2.2.4} and \cite{HTT}*{Definition 5.5.8.18}.} if it is a projective object of $\Mod^{\geq 0}_R$, i.e.\ if the functor
\[
\Hom_R(M,-): \Mod^{\geq 0}_{R} \ra \Spc
\]
commutes with geometric realizations in $\Mod^{\geq 0}_R$.
Thus for any $N \in \Mod^{\geq 1}_R$ consider
\[
C_n(N) = 0\times_{N[1]}0\times_{N[1]}\cdots\times_{N[1]}0
\]
a C\v{e}ch resolution of $N$, i.e.\ $\colim_{\Delta^{\rm op}}C_{\bullet}(N) \simeq N$\footnote{The geometric realization is taken in the category $\Mod^{\geq 0}_R$.}. So we can compute 
\[
\pi_0(\Hom(M,N)) \simeq \pi_0(\Hom_R(M,\left|C_{\bullet}(N)\right|)) \simeq \pi_0(\left|\Hom_R(M,C_{\bullet}(N))\right|) \simeq \pi_0(\ast) \simeq 0.
\]

(2) $\Rightarrow$ (1): For $M \in \Perf(R)$, we notice that the functor
\begin{align*}
\Mod^{\rm disc.}_{\pi_0(R)} \ra & \Spc \\
N \mapsto & \Hom_{\pi_0(R)}(\pi_0(M),N)
\end{align*}
from discrete modules over $\pi_0(R)$ commutes with filtered colimits, since\footnote{Indeed, by the t-structure on $\Mod_R$ one has 
\[
\Hom_R(M,N) \simeq \Fib\left(\Hom_R(\tau_{\geq 1}(M),N) \ra \Hom_R(\tau_{\leq 0}(M)[1],N)\right)
\]
and $\Hom_R(\tau_{\geq 1}(M),N) \simeq 0$, since $N \in \Mod^{\leq 0}_R$.}
\[
\Hom_{\pi_0(R)}(\pi_0(M),N) \simeq \Hom_R(M,N).
\]

Thus $\pi_0(M)$ is finitely generated, so we can pick a morphism
\[
p: R^{\oplus n} \ra M
\]
which induces a surjection on $\pi_0$. Consider $N = \Cofib(R^{\oplus n} \ra M)$, by construction, one has $N \in \Mod^{\geq 1}_R$. So the composite
\[
R^{\oplus n} \ra N
\]
is null-homotopic. Hence, there exists a section of $p$ which exhibits $M$ as a direct summand of $R^{\oplus n}$.
\end{proof}

In light of the previous lemma we define a weight structure on $\Perf(R)$ as follows. Let $\Perf(R)_{w \leq 0}$ be the subcategory of perfect modules of projective amplitude $\leq 0$ and $\Perf(R)_{w \geq 0}$ the subcategory of connective perfect modules.

\begin{lem}
\label{lem:weight-structure-on-Perf}
The categories $(\Perf(R)_{w \leq 0},\Perf(R)_{w \geq 0})$ form a weight structure on $\Perf(R)$. Moreover, the heart $\Perf(R)^{\heartsuit}_{w}$ is the category of finitely generated projective modules.
\end{lem}

\begin{proof}
For condition (i), all we have to notice is that if $M$ has projective amplitude $0$, then $M[-1]$ has projective amplitude $1$, which is clear.

Condition (ii) is true by definition of projective amplitude.

Finally, to check condition (iii) we notice that the inclusion $\Perf(R)_{w \leq 0} \subseteq \Perf(R)$ (resp. $\Perf(R)_{w \geq 0} \subseteq \Perf(R)$) admits a right adjoint $\tau_{w\leq 0}$ (resp. left adjoint $\tau_{w \geq 0}$). And similarly for $\Perf_{w \leq n}$ or $\Perf_{w \geq n}$ for any $n$. Hence, for any $M \in \Perf(R)$ the canonical adjunction maps give
\[
\tau_{w \leq 0}(M) \ra M \ra \tau_{w \geq 1}(M)
\]
whose composite is null-homotopic, by definition of projective amplitude.

The second statement is exactly the content of Lemma \ref{lem:heart-is-projective}.
\end{proof}

\begin{prop}
\label{prop:non-degenerated-and-bounded-weight-structure}
The weight structure on $\Perf(R)$ is non-degenerate and bounded.
\end{prop}

\begin{proof}
First we check that the weight structure is non-degenerate.
Suppose that $M \in \Perf(R)_{w \leq -n}$ for all $n \geq 0$. Then for any $N \in \Perf(R)^{\geq -n+1}$ we have
\[\Hom_R(M,N) \simeq 0.\]
In particular,
\[
\Hom_R(M,R) \simeq M^{\vee} \simeq 0,
\]
so $M \simeq 0$.

Similarly, if $M \in \Perf(R)_{w \leq n}$ for all $n \geq 0$, then for any $N \in \Perf(R)_{w \leq n-1}$
\[
\Hom_R(N,M) \simeq 0.
\]
In particular,
\[
\Hom_R(R,M) \simeq M \simeq 0.
\]

Now we check the boundedness condition. For $M \in \Perf(R)$, by Corollary 7.2.4.5 from \cite{HA} we know that $M \in \Perf(R)^{\geq -n}$ for some $n \in \bN$, i.e.\ $M \in \Perf(R)_{w \geq -n}$ for some $n \in \bN$. Now assume by contradiction that $M$ has unbounded projective amplitude, i.e.\ for all $n\geq 1$, there exists $N \in \Perf(R)^{\geq (n+1)}$ such that
\[
\Hom_R(M,N) \not\simeq 0.
\]
Thus
\[
\Hom_R(N^{\vee},M^{\vee}) \not\simeq 0,
\]
for $N^{\vee} \in \Perf(R)^{\leq -n-1}$. That is a contradiction with $M \in \Perf(R)^{\geq -m}$ for some $m\in \bN$.
\end{proof}

Proposition \ref{prop:non-degenerated-and-bounded-weight-structure} and Theorem \ref{thm:K-theory-heart-of-weight} imply the following result.

\begin{cor}
\label{cor:K-theory-Perf}
Let $R$ be any connective $\Einf$-ring spectrum, then there is an equivalence of K-theory spectra
\[
\K(\Vect(R)) \overset{\simeq}{\ra} \K(\Perf(R))
\]
where $\Vect(R)$ is the category of finitely generated projective modules over $R$.
\end{cor}

\begin{rem}
\label{rem:t-structure-does-not-work}
We notice that the above Corollary doesn't follow from Barwick's theorem of the heart for t-structure even for the ring of dual numbers $R = k[\epsilon]/(\epsilon^2)$ for $|\epsilon| = 1$. Indeed, the category $\Perf(R)$ does not have a bounded t-structure. Consider $k \in \Mod_R$ it follows that for all $n \in \bZ_{\geq 0}$
\[\tau_{\geq -n}(k) \simeq (R[n] \ra \cdots \ra R[0]),\]
which is a compact object of $\Mod^{\geq -n}_{R}$, i.e.\ $k$ is almost perfect. However, 
\[
\Hom(k,-) \simeq \lim_{n \in \bZ_{\geq 0}} \Hom(\tau_{\geq -n}(k),-)
\]
does not commute with filtered colimits, i.e.\ $k$ is not compact in $\Mod_R$.
\end{rem}

\section{Preliminaries on smooth-extended prestacks}
\label{sec:prestacks}

In this section we review basic aspects of the theory of derived algebraic geometry over simplicial commutative rings (for instance as developed in \cite{SAG}*{\S 25.1.1}). We discuss the concept of a prestack (and stack) left Kan extended from smooth algebras and we prove that the stack of vector bundles over an arbitrary commutative ring $k$ is left Kan extended from smooth algebras (Theorem \ref{thm:sVect-is-smooth-extended}). Together with the observation that over an ordinary commutative ring $k$ all smooth simplicial $k$-algebras are discrete, i.e.\ have vanishing positive homotopy groups, this result will be important to bootstrap the usual determinant from finitely generated modules over an ordinary ring to the category of finitely generated (derived) modules over a genuinely derived ring.

\subsection{Stacks smooth extended}
\label{subsec:stacks-smooth-extended}

Consider $k$ a commutative commutative ring\footnote{More generally, one could consider $k$ a simplicial commutative ring.} let $\oAlg^{\rm Poly.}_{k}$ denote the ordinary category whose objects are finite polynomial rings over $k$ and morphisms $k$-algebra homomorphisms. 

We define the category of simplicial commutative $k$-algebras to be
\[
\sAlg_k := \left\{F: \N\oAlg^{\rm Poly.}_{k} \ra \Spc \; | \; F \mbox{ preserves finite products } \right\}.
\]

We refer the reader to \cite{SAG}*{Remark 25.1.1.3} for a discussion of why this category recovers the usual (i.e.\ commutative ring objects in simplicial sets) definition of simplicial commutative $k$-algebras.

A prestack over $k$ is a functor
\[
\sX: \sAlg_k \ra \Spc.
\]
We will say that $\sX$ is a stack if $\sX$ satisfies \'etale descent.

\begin{notation}
We will denote by $\Schaff$ the opposite of the category $\sAlg_k$, the subcategory of classical affine schemes will be denote by $\Schaffcl$. Given a prestack $\sX: (\Schaff)^{\rm op} \ra \Spc$ we denote $\classical{\sX}: (\Schaffcl)^{\rm op} \ra \Spc$ its restriction to classical affine schemes.
\end{notation}

Given a map $f:R \ra R'$ in $\sAlg_k$ we say that $f$ is \emph{formally smooth} if the relative cotangent complex $\bL_{R'/R}$ is concentrated in degrees $\leq -1$, i.e.
\[
\bL_{R'/R} \in \Mod(R)^{\leq -1}.
\]

Moreover, one says that $f:R \ra R'$ is of \emph{finite presentation} if the functor
\[
\Hom_{(\sAlg_{k})_{R/}}(-,R'): (\sAlg_{k})_{R/} \ra \Spc
\]
preserves filtered colimits, i.e.\ $R'$ is a compact object of $(\sAlg_{k})_{R/}$, the category of simplicial commutative $k$-algebras under $R$.

\begin{defn}
For a map $f:R \ra R'$ we say that $f$ is \emph{smooth} if it is formally smooth and of finite presentation. Let $\sAlg^{\rm sm}_k$ denote the subcategory of $\sAlg_k$ consisting of smooth simplicial commutative $k$-algebras.
\end{defn}

\begin{notation}
To alleviate the notation we introduce the following:
\begin{align*}
    \smooth{(-)}: \PStk & \ra \Fun(\sAlg^{\rm sm}_k,\Spc) \\
    \sX & \mapsto \smooth{\sX} := \left.\sX\right|_{\sAlg^{\rm sm}_k}
\end{align*}
and
\begin{align*}
    \LKEsm: \Fun(\sAlg^{\rm sm}_k,\Spc) & \ra \PStk \\
    \sX_0 & \mapsto \LKE_{\sAlg^{\rm sm}_k \hra \sAlg_k}(\sX_0).
\end{align*}
\end{notation}

Given a prestack $\sX \in \PStk$ one has a canonical map
\begin{equation}
    \label{eq:defn-map-smooth-extended-prestacks}
    \LKEsm(\smooth{\sX}) \ra \sX.
\end{equation}

\begin{defn}
A prestack $\sX$ is said to be \emph{smooth-extended} if the map (\ref{eq:defn-map-smooth-extended-prestacks}) is an isomorphism.
\end{defn}

Similarly, given $\sX$ a stack, one has a canonical map
\begin{equation}
    \label{eq:defn-map-smooth-extended-stacks}
    \LLKEsm(\smooth{\sX}) \ra \sX,
\end{equation}
where $\LLKEsm$ means we take the fppf sheafification after the left Kan extension.

\begin{defn}
For a stack $\sX$, one says that $\sX$ is smooth-extended if the canonical map (\ref{eq:defn-map-smooth-extended-stacks}) is an equivalence.
\end{defn}

In the next section we will prove that the stack of vector bundles is smooth-extended. For that we will use a technical observation from A.\ Mathew. To formulate it we first introduce the following notion.

Given a map $f:R \ra R'$ of simplicial commutative $k$-algebras, one says that $f$ is a \emph{henselian surjection} if
\begin{enumerate}[(i)]
    \item $\pi_0(f)$ is a surjection; and
    \item $(R,\pi_0(f))$ is a henselian pair\footnote{See \cite{stacks-project}*{Tag 09XE} for a definition.}.
\end{enumerate}

\begin{prop}[\cite{EHKSY}*{Proposition A.0.1}]
\label{prop:Mathew-criterion}
Given a prestack $\sX: \sAlg_k \ra \Spc$ suppose that:
\begin{enumerate}[(i)]
    \item $\sX$ preserves filtered colimits;
    \item for every henselian surjection $A \ra B$, the map $\sX(A) \ra \sX(B)$ is an effective epimorphism\footnote{I.e. $\sX(B)$ is the colimit of $\{\sX(A^{\times^n_B})\}_{n\in \Delta^{\rm op}}$.};
    \item for every henselian surjection $A \ra C \la B$, the square
    \[
    \begin{tikzcd}
    \sX(A\underset{C}{\times}B) \ar[r] \ar[d] & \sX(B) \ar[d] \\
    \sX(A) \ar[r] & \sX(C)
    \end{tikzcd}
    \]
    is Cartesian.
\end{enumerate}
Then the canonical map
\[
\LKEsm(\smooth{\sX}) \ra \sX
\]
is an isomorphism.
\end{prop}

\subsection{Application: the prestack of vector bundles}

In this section we prove that the stack of vector bundles is smooth-extended. We start with a review of the stack of vector bundles. Then we describe this stack as the colimit of a disjoint union of classifying spaces. Using the criterion from the previous section we prove that this colimit is smooth-extended, which gives our result\footnote{The reader is invited to see \cite{EHKSY}*{Appendix A} for arguments that prove that $\sVect$ is smooth-extended as a \emph{prestack}.}.

For $S$ an affine scheme, we let $\QCoh(S)$ denote its category of quasi-coherent sheaves, and $\QCoh(S)^{\geq 0}$ the subcategory of connective objects. One way to define vector bundles over a derived ring is as\footnote{Recall that dualizable means that there exists $\sG \in \QCoh(S)^{\geq 0}$ and maps $e:\sF\otimes \sG \ra \sO_S$ and $\epsilon: \sO_S \ra \sG\otimes \sF$, s.t.\ 
\[
\sF \simeq \sF\otimes \sO_S \overset{\id_{\sF}\otimes\epsilon}{\ra} \sF\otimes\sG \otimes \sF \overset{e\otimes\id_{\sF}}{\ra} \sO_S\otimes \sF \simeq \sF
\]
is isomorphic to $\id_{\sF}$. One also ask a similar condition for $\sG$.}
\[\Vect(S) = \{\sF \in \QCoh(S)^{\geq 0} \; | \; \sF \; \mbox{is dualizable in }\QCoh(S)^{\geq 0}\}\]

For our convenience we recall some equivalent characterizations of $\Vect(S)$. 

\begin{prop}
\label{prop:equivalent-conditions-projective-modules}
For $\sF \in \QCoh(S)^{\geq 0}$ the following are equivalent:
\begin{enumerate}[(1)]
    \item $\sF$ is a retract of $\sO^{\oplus n}_{S}$, for some $n\geq 1$;
    \item $\sF$ is dualizable in $\QCoh(S)^{\geq 0}$;
    \item $\sF$ is compact and projective as an object of $\QCoh(S)^{\geq 0}$, i.e. $\Hom_R(\sF,-)$ commutes with sifted colimits\footnote{The condition of being compact corresponds to $\Hom(\sF,-)$ commuting with filtered colimits, whereas projective means $\Hom(\sF,-)$ commutes with geometric realizations. By \cite{HTT}*{Lemma 5.5.8.14.} any sifted colimit is a combination of those.};
    \item $\sF$ is flat and almost perfect;
    \item $\sF$ is flat and $\pi_0(\sF)$ is locally free with respect to $\sO_{\classical{S}}$\footnote{I.e.\ $\pi_0(\sF)$ is a vector bundle over $\classical{S}$.}.
\end{enumerate}
\end{prop}

\begin{proof}
The equivalence between (1) and (2) is \cite{SAG}*{Proposition 2.9.1.5.}, the equivalence between (1) and (3) is \cite{HA}*{Proposition 7.2.2.7.} and the equivalence between (1) and (4) is \cite{HA}*{Proposition 7.2.4.20.}. 

(4) $\Rightarrow$ (5): Notice that by definition $\sF$ is flat if and only if $\pi_0(\sF)$ is flat and $\pi_k(\sF)\simeq \pi_k(\sO_S)\otimes_{\sO_{\classical{S}}}\pi_0(\sF)$ for all $k \geq 0$. Moreover, since $\sF$ is almost perfect and connective, we have that $\tau_{\leq 0}\sF \simeq \pi_0(\sF)$ is compact in $\QCoh(S)^{\heartsuit}$, thus finitely presented as a $\sO_{\classical{S}}$ sheaf. Now, it is a classical fact \cite{Rotman}*{Theorem 3.56.} that a flat and finitely presented module is projective, thus $\pi_0(\sF)$ is projective.

(5) $\Rightarrow$ (1): Consider $\sN = \Hom(\sF,\sO_S)$, we can compute
\[\pi_i\sN \simeq \Ext^i_{S}(\sF,\sO_S) \simeq \Tor^i_S(\sF^{\vee},\sO_S),\]
since $\sF^{\vee}$ is also flat, one has
\[\pi_{i}\sN \simeq \pi_0(\sF)^{\vee}\otimes_{\pi_0(\sO_S)}\pi_i(\sO_S)\]
which vanishes for $i < 0$, because $\sO_S$ is connective.
\end{proof}

\begin{rem}
\label{rem:projective-commutes-with-geometric-realization}
For $S$ an affine scheme, recall $\sF \in \QCoh(S)$ is a projective object if the functor $\Hom_S(\sF,-)$ commutes with geometric realizations. Moreover, if $\sF$ is finitely generated, then $\Hom_S(\sF,-)$ commutes with sifted colimits. In particular, for any $\sF \in \Vect(S)$, the functor $\Hom_S(\sF,-)$ commutes with sifted colimits.
\end{rem}

The main object of interest for us is the prestack
\begin{align*}
    \sVect: \Schaffop & \ra \Spc \\
    S & \mapsto \Vect(S)^{\simeq}.
\end{align*}

We first observe the following standard fact.

\begin{lem}
\label{lem:sVect-is-a-stack}
The prestack $\sVect$ satisfies descent with respect to the flat topology, i.e. it is a stack.
\end{lem}

\begin{proof}
Since the assignment $S \mapsto \QCoh(S)$ satisfies descent with respect to the flat topology (see \cite{SAG}*{Corollary D.6.3.3.}) and $S \mapsto \QCoh(S)^{\geq 0}$ satisfies descent as well, by \cite{SAG}*{Theorem D.6.3.1}. It is enough to check that the condition of being a vector bundle, i.e.\ to be dualizable, is local with respect to the flat topology. This is exactly \cite{SAG}*{Proposition 2.9.1.4}.
\end{proof}

Before giving the proof that the stack of vector bundles is smooth-extended we introduce an auxiliary stack that we can describe more concretely than $\sVect$. For each $n\geq 1$ let
\[
\GL_n: \sAlg_k \ra \Spc
\]
be defined as the pullback
\begin{equation}
\label{eq:pullback-square-GLn}
    \begin{tikzcd}
    \GL_n(R) \ar[r] \ar[d] & \M_n(R) \ar[d] \\
    \GL_n(\pi_0(R)) \ar[r] & \M_n(\pi_0(R)) 
    \end{tikzcd}
\end{equation}
where $\M_n(R)$ is the $n^2$-affine space over $R$ of $n \times n$-matrices. We notice that $\GL_n$ is represented by an affine scheme over $k$, namely
\[
\GL_n(R) \simeq \Hom_{\sAlg_k}(A_n,R),
\]
where 
\[
A_n \simeq k[(x_{i,j})_{1 \leq i,j \leq n}]/[\det(x_{i,j})^{-1}],
\]
where $\det(x_{i,j})$ is the determinant formula for the entries of an $n$ by $n$ matrix.

\begin{prop}
\label{prop:GL_n-is-smooth-extended}
For any $n \geq 1$ the canonical map
\begin{equation}
    \label{eq:smooth-extended-functor-satisfies-descent}
    \LLKEsm(\smooth{\GL_n}) \ra \GL_n
\end{equation}
is an equivalence.
\end{prop}

\begin{proof}
We will apply Proposition \ref{prop:Mathew-criterion}. 

We notice that in the defining diagram (\ref{eq:pullback-square-GLn}) the terms $\GL_n(\pi_0(R))$, $\M_n(\pi_0(R))$ and $\M_n(R)$ all commute with filtered colimits, since filtered colimits commute with finite limits, so condition (i) holds for $\GL_n(R)$.

Since $\GL_n$ is representable by \cite{Gruson}*{Th\'eor\`eme I.8} for any henselian surjection $f:R \ra R'$ the map $\GL_n(R) \ra \GL_n(R')$ has a section, which implies that it is an effective epimorphism.

For condition (iii), we notice that since $\GL_n$ is representable by $A_n$ it commutes with limits. 

Thus, the canonical map
\[
\LKEsm(\smooth{GL_n}) \ra \GL_n
\]
is an equivalence.

The equivalence of (\ref{eq:smooth-extended-functor-satisfies-descent}) follows from the fact that $\GL_n$ satisfies flat descent, i.e.\ the map
\[
\LLKEsm(\smooth{\GL_n}) \ra \LLKEsm(\smooth{\GL_n})
\]
is an equivalence.
\end{proof}

Now consider the prestack defined by
\begin{align*}
    \Schaffop & \ra \Spc \\
    S & \mapsto \BGL(S) = \left|\B_{\bullet}\GL_n(S)\right|
\end{align*}
where $\B_{\bullet}\GL_n(S)$ is the simplicial object whose value on $[m] \in \Delta^{\rm op}$ consists of $m$ copies of $\GL_n(S)$ and maps are the canonical projections and inner multiplications. Let $\sBGL_n$ denote its fppf sheafification, thus we have the stack
\[\sBGL = \bigsqcup_{n\geq 0}\sBGL_n.\]

\begin{prop}
\label{prop:BGL-is-classical}
The canonical map
\[
\LKEclStk(\classical{\sBGL}) \ra \sBGL
\]
is an equivalence of stacks.
\end{prop}

\begin{proof}
Fix $n\geq 0$, we first notice that passing to classifying spaces, is given by a geometric realization, which commutes with Left Kan extensions. Thus, from Proposition \ref{prop:GL_n-is-smooth-extended} one has
\begin{equation}
\label{eq:BGL_n-is-classical}
\B(\LKEsm(\smooth{\GL_n})) \simeq \LKEsm(\B(\smooth{\GL_n})) \simeq \LKEsm(\smooth{\BGL_n}) \overset{\simeq}{\ra} \BGL_n,   
\end{equation}
since $\smooth{(-)}$ also commutes with geometric realizations.

If we sheafify both sides of (\ref{eq:BGL_n-is-classical}) one obtains an equivalence
\[
\LLKEsm(\smooth{\BGL_n}) \simeq \LLKEsm(\smooth{\sBGL_n}) \overset{\simeq}{\ra}\sBGL_n,
\]
where the last somorphism follows from the fact that sheafification is also given by a colimit.

Finally, since $\LLKEsm$ and $\smooth{}$ commute with disjoint unions, we are done.
\end{proof}

We notice that there are canonical maps of prestack
\[\BGL_n \ra \sVect\]
that send the unique $S$-point of $\BGL_n$ to $\sO^{\oplus n}_{S}$. In particular, these assemble to a map of prestacks
\begin{equation}
\label{eq:map-between-BGL-and-sVect}
\BGL \ra \sVect.    
\end{equation}

By Lemma \ref{lem:sVect-is-a-stack} the map (\ref{eq:map-between-BGL-and-sVect}) factors through the sheafification of $\BGL$:
\[f:\sBGL \ra \sVect.\]

\begin{lem}
\label{lem:Vect-is-BGL}
The map
\begin{equation}
    \label{eq:BGL-to-sVect}
    f: \sBGL \ra \sVect
\end{equation}
is an equivalence of stacks.
\end{lem}

\begin{proof}
We need to check that for any $R \in \sAlg_k$ and $\sF \in \sVect(R)$, there exists a flat cover $R \ra R'$ such that $\sF\otimes_{R}R'$ is a finitely generated free $R'$-module.

Proposition 2.9.2.3 from \cite{SAG} says that we can do even better, that is even Zariski locally any vector bundle is trivial. More precisely, the result states that there exists a collection of elements $x_1,\ldots,x_m \in \pi_0(R)$ such that
\begin{enumerate}[(i)]
    \item $(x_i)$ generate the unit ideal of $R$, i.e.\ $\{R \ra R[x^{-1}_i]\}_{i \in I}$ is a Zariski cover; and
    \item $\sF\otimes_{R}R[x^{-1}_i]$ is finitely generated and free for every $i$.
\end{enumerate}

Since any Zariski sheaf is also an fppf sheaf, this finishes the proof.
\end{proof}

\begin{thm}
\label{thm:sVect-is-smooth-extended}
The canonical map
\[
\LLKEsm(\smooth{\sVect}) \ra \sVect
\]
is an equivalence of stacks.
\end{thm}

\begin{proof}
We notice that we have the following commutative diagram
\[
\begin{tikzcd}
\LLKEsm(\smooth{\sBGL}) \ar[r,"\simeq"] \ar[d,"\simeq"] & \sBGL \ar[d,"\simeq"] \\
\LLKEsm(\smooth{\sVect}) \ar[r,"\simeq"] & \sVect.
\end{tikzcd}
\]
Here the right vertical arrow is an equivalence because of Lemma \ref{lem:Vect-is-BGL}, the left vertical arrow is an equivalence because $\LLKEsm\circ\smooth{(-)}$ is a functor. Finally, the top horizontal arrow is an equivalence by Lemma \ref{prop:BGL-is-classical}.
\end{proof}

We consider the stacks of line bundles $\sPic$ defined as
\begin{align*}
    \sPic(R) := \left\{ \sF \in \sVect(R) \;| \; \sF \; \mbox{is locally free of rank 1} \right\}
\end{align*},
where locally free of rank $1$ means that after a Zariski base change as described in the proof of Lemma \ref{lem:Vect-is-BGL} $\sF_i$ is finitely generated of rank $1$. 

We will also consider the stack of graded line bundles $\sPicgr$ defined as
\[
\sPicgr(R) := \left\{\sF \in \Mod(R) \; | \; \sF \mbox{ is invertible with respect to the symmetric monoidal structure}\;\right\}.
\]

\begin{cor}
\label{cor:graded-Pic-classical}
The canonical maps
\[
\LLKEsm(\smooth{\sPic}) \ra \sPic \;\;\; \mbox{and} \;\;\;
\LLKEsm(\smooth{\sPicgr}) \ra \sPicgr
\]
are equivalences of stacks.
\end{cor}

\begin{proof}
For the first statement we notice that the restriction of (\ref{eq:BGL-to-sVect}) to $\sBGL_1$ factors through $\sPic$ and the induced map
\[
\sBGL_1 \ra \sPic
\]
is an equivalence of stacks. Indeed, by Proposition 2.9.4.2 a line bundle, i.e.\ invertible object of $\Vect(R)$ is a quasi-coherent sheaf that is locally free of rank $1$. The claim now follows from the proof of Proposition \ref{prop:BGL-is-classical}.


For the second statement we notice that by \cite{SAG}*{Remark 2.9.5.8.} one has an equivalence
\[
\sPicgr \simeq \sPic \times {^{\rm L, et}(\underline{\bZ})},
\]
where ${^{\rm L, et}(\underline{\bZ})}$ denotes the \emph{\'etale}\footnote{Indeed, by a result of To\"en (see \cite{Toen-descent-n-stacks}*{Theor\`eme 2.1}), the \'etale and fppf sheafification agree.} sheafification of the constant presheaf with value $\bZ$.  
Thus, it is enough to check that the map
\[
\LLKEsm(\smooth{({^{\rm L, et}(\underline{\bZ})})}) \ra {^{\rm L, et}(\underline{\bZ})}
\]
is an isomorphism.

By \cite{HA}*{Theorem 7.5.0.6.} one has an equivalence
\[
\sAlg_{k}^{\rm et} \simeq {^{\cl}\sAlg_k}^{\rm et},
\]
between the category of \'etale $k$-algebras and the nerve of the ordinary category of classical \'etale $k$-algebras, since we supposed $k \simeq \pi_0(k)$. Similarly (see \cite{HAG-II}*{Theorem 2.2.2.6}), we have
\[
\sAlg^{\rm sm}_k \simeq \classical{\sAlg^{\rm sm}_{k}}.
\]
Finally, since the \'etale and smooth topologies on classical $k$-algebras give the same notion of sheaves, one has
\[
{^{L,\rm et, c\ell}(\underline{\bZ})} \simeq {^{L,\rm sm, c\ell}(\underline{\bZ})}
\]
where ${^{L,\rm et, c\ell}(\underline{\bZ})}$ denotes the \'etale sheafification of the classical prestack into a classical stack, and ${^{L,\rm sm, c\ell}(\underline{\bZ})}$ denotes the smooth sheafification. So, we obtain
\[
\LLKEsm(\smooth{({^{L,\rm sm, c\ell}(\underline{\bZ})})}) \simeq {^{\rm L, et}(\underline{\bZ})},
\]
where tautological we have
\[
\smooth{({^{L,\rm sm, c\ell}(\underline{\bZ})})} \simeq \smooth{({^{\rm L, et}(\underline{\bZ})})}.
\]
\end{proof}

\section{Determinant map for perfect complexes}
\label{sec:determinant}

In this section we construct the main input for the determinant map of Tate objects. Namely, we construct a determinant map for the prestack of perfect complexes. We follow the ideas of \cite{STV}.

\subsection{Recollection on the determinant of classical rings}

In this section $R_0$ will always denote an ordinary commutative ring. We let $\K_{\rm Spc}(R_0)$ denote the K-theory space of $\Vect(R_0) = \N(\oModfgp_{R_0})$, i.e.\ the nerve of the ordinary exact category of finitely generated projective $R_0$-modules. We denote by ${^{\rm c}\sK}$ the classical prestack associated to this construction.

\begin{rem}
\label{rem:K-theory-space-is-not-a-prestack}
Notice that ${^{\rm c}\sK}$ is not a classical stack. The main reason is that it doesn't satisfy \'etale descent. However if one constructs connective K-theory as a presheaf of connective spectra, then it does satisfy Zariski descent (when restricted to qcqs schemes, by \cite{TT}*{Theorem 8.1}) however it still does not satisfy \'etale (nor flat) descent (cf.\ \cite{Thomason-etale}).
\end{rem}

We also consider ${^{\rm c}\sPic}$ the classical prestack that sends an ordinary ring $R_0$ to the underlying $\infty$-groupoid of the subcategory $\Pic(R_0) \subset \Vect(R_0)$ of invertible elements, i.e. $\sL \in \Pic(R_0)$ such that there exists a dual $\sL^{\vee}$ and the evaluation (equivalently, coevaluation map)
\[
\sL^{\vee}\otimes \sL \ra R_0
\]
is an equivalence. 

From the descent of $\sVect(R_0)$ and the fact that the invertibility condition is local for the flat topology we have that ${^{\rm c}\sPic}$ is a classical stack. It is actually a classical group stack\footnote{I.e. a group object in the category of classical stacks.}, with group structure given by the restriction of the tensor product on $\Vect(R_0)$. More generally, we will be interested in the group stack ${^{\rm c}\sPicgr}$ whose $R_0$-points are the underlying groupoid of category of graded lines over $R_0$\footnote{More formally we define ${^{\rm c}\sPicgr}(R_0)$ to be the sub $\infty$-groupoid of $(\Mod_{R_0})^{\simeq}$ spanned by its invertible elements.}, with tensor product that takes into account the degree.

The following is a well-known fact.

\begin{thm}
\label{thm:determinant-K-theory-of-ordinary-rings}
There exists a map of classical prestacks
\[
{^{\rm c}D}: {^{\rm c}\sK} \ra {^{\rm c}\sPicgr}.
\]
Moreover, the truncation of $^{\rm c}{D}$ to a map of $1$-groupoids\footnote{By that we mean the composition of the functor defined by each prestack with the truncation functor $\tau_{\leq 1}:\Spc \ra \Spc^{\leq 1}$ from $\infty$-groupoids to $1$-groupoids, i.e. the left adjoint to the natural inclusion.} followed by Zariski sheafification is an isomorphism.
\end{thm}

\begin{proof}
The first claim follows from \cite{Bhatt-Scholze}*{Corollary 12.17} by passing to the underlying space of their functorial map of connective spectra.

The second claim is \cite{Bhatt-Scholze}*{Proposition 12.18}.
\end{proof}


For our goal in the next section, it will be convenient to make explicit the construction that produces the map ${^{\rm c}D}$ from Theorem \ref{thm:determinant-K-theory-of-ordinary-rings}. For $R_0$ an ordinary commutative ring, we let $\Vect(R_0)^{\simeq}$ denote its groupoid of finitely generated projective modules, and $\Pic^{\gr}(R_0)$ denote the groupoid of graded line bundles. Both $(\Vect(R_0)^{\simeq},\oplus)$ and $(\Pic^{\gr}(R_0),\otimes_{R_0})$ are symmetric monoidal groupoids. 


The following is a standard fact.

\begin{lem}
\label{lem:det-as-symmetric-monoidal-functor}
One has a symmetric monoidal functor
\begin{equation}
\label{eq:det-symmetric-monoidal-functor}
\det^{\gr}: \Vect(R_0)^{\simeq} \ra \Pic^{\gr}(R_0)
\end{equation}
given by 
\[\det^{\gr}(M) = \left(\bwedge{\rank (M)}M,\rank(M)\right).\]
\end{lem}

For the next section, we will need a rephrasing of Lemma \ref{lem:det-as-symmetric-monoidal-functor}.

The assignment
\[
[n] \in \Delta^{\rm op} \; \mapsto \; (\Vect(R_0)^{\simeq})^{\times n} \in \Grpd
\]
defines a simplicial object in $\Grpd$, the $(2,1)$-category of groupoids. More precisely, one has weak functors\footnote{See \cite{stacks-project}*{Tag 003N} for what we concretely mean by that.}
\[
\B_{\bullet}\Vect(R_0): \Delta^{\rm op} \ra \Grpd,
\]
and
\[
\B_{\bullet}\Pic^{\gr}(R_0): \Delta^{\rm op} \ra \Grpd.
\]

\begin{lem}
\label{lem:det-lax-natural-transformation}
The symmetric monoidal functor (\ref{eq:det-symmetric-monoidal-functor}) determines a weak natural transformation\footnote{We refer the reader to \cite{Sketches-I}*{\S B.1} for a definition of this. Notice that in \cite{Sketches-I} this is called a pseudo-natural transformation.}
\[
\det^{\gr}_{\bullet,R_0}: \B_{\bullet}\Vect(R_0) \ra \B_{\bullet}\Pic^{\gr}(R_0)
\]
functorial in $R_0$.
\end{lem}

\begin{rem}
\label{rem:equivalence-fibered-groupoids}
From \cite{Sketches-I}*{\S B.3 Theorem 1.3.6} one knows that there exists an equivalence between the strict $2$-category $\Fun(\Delta^{\rm op},\Grpd)$ of weak functors and weak natural transformations between them and the $2$-category of categories fibered in groupoids over $\Delta^{\rm op}$.

The construction that sends a weak functor $F:\Delta^{\rm op} \ra \Grpd$ to the category fibered in groupoids over $\widetilde{F} \ra \Delta^{\rm op}$ is the Grothendieck construction. Informally, the objects of $\widetilde{F}$ are $\sqcup_{[n] \in \Delta^{\rm op}}\widetilde{F}([n])$ and morphisms between $([n],x)$ and $([m],y)$, where $x \in F([n])$ and $y \in F([m])$, are given by
\[
\Hom_{\widetilde{F}}(([n],x),([m],y)) = \bigsqcup_{\alpha \in \Hom_{\Delta^{\rm op}}([n],[m])}\Hom_{F([m])}(F(\alpha)(x),y).
\]
\end{rem}

By Remark \ref{rem:equivalence-fibered-groupoids} one has a map
\begin{equation}
\label{eq:det-map-of-fibered-categories}
\widetilde{\det^{\gr}_{R_0}}: \widetilde{\B\Vect(R_0)} \ra \widetilde{\B\Pic(R_0)}    
\end{equation}
of categories fibered in groupoids over $\Delta^{\rm op}$.

The following is a compatibility result that we will need in the construction of the determinant as a map of simplicial objects in $\infty$-groupoids.

\begin{lem}
\label{lem:B-and-N-are-compatible}
Consider $\b{C}$ an (ordinary) symmetric monoidal groupoid, and let $\B_{\bullet}\b{C}: \Delta^{\rm op} \ra \Grpd$ be the simplicial object\footnote{This is a weak functor into the $(2,1)$-category of groupoids.} encoding its symmetric monoidal structure.
Then one has an equivalence of symmetric monoidal $\infty$-groupoids\footnote{Strictly speaking, the nerve functor on the lefthand side is a coherent nerve from $\Grpd$ into $\Spc$, and not simply the nerve functor from $1$-categories to categories.}
\[
\N(\B_{\bullet}\b{C}) \simeq \B_{\bullet}\N(\b{C}),
\]
here the righthand side is the simplicial object in $\Spc$ given by the Bar construction applied to $\N(\b{C})$.
\end{lem}

\begin{construction}
\label{cons:det-as-map-of-right-fibrations}
By \cite{HTT}*{Proposition 2.1.1.3} one has that
\[
\N(\widetilde{\B\Vect(R_0)}) \ra \N\Delta^{\rm op} \;\;\; \mbox{and}\;\;\; \N(\widetilde{\B\Pic(R_0)}) \ra \N\Delta^{\rm op}
\]
are right fibrations of categories.

Moreover, the map $\widetilde{\det^{\gr}_{R_0}}$ induces a map of right fibrations $\N(\widetilde{\det^{\gr}_{R_0}})$. 

By Lemma \ref{lem:B-and-N-are-compatible} we notice that
\[\N(\widetilde{\B\Vect(R_0)}) \simeq \Un(\B_{\bullet}\N(\Vect(R_0))),
\;\;\;
\mbox{and}
\;\;\;
\N(\widetilde{\B\Pic(R_0)}) \simeq \Un(\B_{\bullet}\N(\Pic(R_0)))
\]
are isomorphic as right fibrations over $\N(\Delta^{\rm op})$, where $\Un(\B_{\bullet}\N(\Vect(R_0)))$ denotes the right fibration over $\N\Delta^{\rm op}$ obtained from the functor $\B_{\bullet}\N(\Vect(R_0)):\Delta^{\rm op} \ra \Spc$\footnote{See \cite{HTT}*{}.}; and similarly for $\N(\Pic(R_0))$.

Let
\begin{equation}
    \widetilde{\sDet}_{R_0}: \Un(\B_{\bullet}\N(\Vect(R_0))) \ra \Un(\B_{\bullet}\N(\Pic(R_0)))
\end{equation}
denote the map of right fibrations induced by $\N(\widetilde{\det^{\gr}_{R_0}})$. We notice that this map is functorial in $R_0$, since all the constructions were functorial.

By applying the straightening functor, one obtains the map
\begin{equation}
\label{eq:det-as-map-of-simplicial-objects-in-oo-categories}
\sDet_{\bullet,R_{0}}: \B_{\bullet}\N(\Vect(R_0)) \ra \B_{\bullet}\N(\Pic(R_0)).    
\end{equation}

In particular, one can restrict the functors above to the category of smooth algebras, we will denote the resulting functor\footnote{We use the same notation $\smooth{(-)}$ for the restriction of a prestack and a classical prestack to a functor from the category of smooth $k$-algebras and classical smooth $k$-algebras. Hopefully this will not cause confusion.}
\begin{equation}
    \label{eq:det-as-map-of-simplicial-objects-in-oo-categories-for-smooth-algebras}
    \sDet_{\bullet,R_{0}}: \B_{\bullet}\N(\smooth{\Vect}(R_0)) \ra \B_{\bullet}\N(\smooth{\Pic}(R_0)).
\end{equation}

Thus, by taking the loop space of the geometric realization\footnote{I.e. a concrete model for the group completion of our $\Einf$-monoid in spaces.} of (\ref{eq:det-as-map-of-simplicial-objects-in-oo-categories}) one obtains the map
\[
{^{\rm c}D}: {^{\rm c}\sK}(R_0) \ra {^{\rm c}\sPic}(R_0).
\]
\end{construction}

\begin{rem}
\label{rem:compatibility-with-Bhatt-Scholze}
It is clear that the map from Construction \ref{cons:det-as-map-of-right-fibrations} agrees with the one on Proposition 12.3 from \cite{Bhatt-Scholze}*{\S 12}.
\end{rem}

\subsection{Construction of determinant for derived rings}

In this section we fix a base commutative ring $k$.

The main goal of this section is to prove the following

\begin{thm}
\label{thm:determinant-of-perfect-complexes}
Let $R$ be a connective dg-algebra over $k$, there exists a map of group-like $\Einf$-monoids in spaces
\[
\sD^{\rm K}(R): \K(\Perf(R)) \ra \sPicgr(R)
\]
functorial in $R$.

In particular, one obtains a map of prestacks over $k$
\begin{equation}
\label{eq:the-determinant-map}
\sD: \sPerf \ra \sPicgr,    
\end{equation}
where $\sPerf$ denotes the prestack associated to $\Perf$, i.e.\ $\sPerf(R) := \Perf(R)^{\simeq}$.
\end{thm}

\begin{proof}
The proof consists of a couple of steps.

\textbf{Step 1.} By Corollary \ref{cor:K-theory-Perf} we have an equivalence $\K(\Vect(R)) \ra \K(\Perf(R))$, thus it is enough to construct a map out of $\K(\Vect(R))$.

\textbf{Step 2.} By Theorem \ref{thm:additive-K-theory-agrees-when-split-cofibrations} we notice that
\[\Ka(\Vect(R)) \overset{\simeq}{\ra} \K(\Vect(R))\]
is an equivalence. We also notice that for a Picard $\infty$-groupoid\footnote{I.e. a symmetric monoidal $\infty$-groupoid where all objects are invertible.} $\sP$ one has an equivalence
\[\sP \overset{\simeq}{\ra} \Ka(\sP)\]
as group-like $\Einf$-monoids in spaces. Indeed, by \cite{HA}*{Definition 5.2.6.2} an $\Einf$-monoid $\sP$ is grouplike if the canonical map
\[(m,p_2): \sP \times \sP \ra \sP \times \sP\]
is an equivalence. Since any $X \in \sP$ has an inverse $X^{-1}$, we notice that the map $\varphi: \sP \times \sP \ra \sP \times \sP$ given by
\[
\varphi(X,Y) = (X\otimes Y^{-1},Y)
\]
is an inverse to $(m,p_2)$. Since $\Ka(\sP)$ is the group completion of $\sP^{\simeq} \simeq \sP$ this gives the claim.


Thus, it is enough to construct a map of simplicial objects in spaces
\begin{equation}
\label{eq:map-of-simplicial-objects-that-gives-additive-K-theory-map}
\B_{\bullet}\Vect(R) \ra B_{\bullet}\sPic^{\gr}(R).    
\end{equation}

Indeed, by taking loop space on the geometric realization of (\ref{eq:map-of-simplicial-objects-that-gives-additive-K-theory-map}) one obtains
\[
\K(\Vect(R)) \simeq \Ka(\Vect(R)) \simeq \Omega\left|\B_{\bullet}\Vect(R)\right| \ra \Omega\left|B_{\bullet}\sPic^{\gr}(R)\right| = \K(\sPic^{\gr}(R)) \simeq \sPic^{\gr}(R).
\]

\textbf{Step 3.} By Theorem \ref{thm:sVect-is-smooth-extended} and the fact that $\LLKEsm$ commutes with products for each $[n] \in \Delta^{\rm op}$ there is an equivalence
\[
\LLKEsm(\B_{n}\smooth{\sVect}(R)) \simeq \B_n\sVect(R)
\]
and similarly for $\sPic^{\gr}(R)$ by Corollary \ref{cor:graded-Pic-classical}. 

Since
\[
\B_{n}\smooth{\sVect}(R) \simeq \B_n\smooth{\Vect}(\pi_0(R)),
\]
one can take the sheafification of the left Kan extension of (\ref{eq:det-as-map-of-simplicial-objects-in-oo-categories-for-smooth-algebras}) from smooth $k$-algebras to all $k$-algebras to obtain the map
\begin{equation}
\label{eq:simplicial-determinant}
\sD_{\bullet}: \B_{\bullet}\sVect(R) \ra \B_{\bullet}\sPic^{\gr}(R).    
\end{equation}

Taking the geometric realization of (\ref{eq:simplicial-determinant}) and loop spaces yields $\sD^{\rm K}(R)$. 

It is clear from the construction that $\sD^{\rm K}$ is functorial and the map $\sD$ is obtained by precomposing $\sD^{\rm K}$ with $\imath: \sPerf(R) \ra \K(\sPerf(R))$ from (\ref{eq:imath-for-spectra}).
\end{proof}

Suppose that $\sX$ is any derived stack, let
\[\sPerf(\sX) = \Maps_{\Stk}(\sX,\sPerf).\]

By applying the map (\ref{eq:the-determinant-map}) we obtain a map of stacks
\begin{equation}
\label{eq:determinant-for-sX}
    \sD_{\sX}: \sPerf(\sX) \ra \sPicgr(\sX).    
\end{equation}

\begin{rem}
The result of Theorem \ref{thm:determinant-of-perfect-complexes} appears in a similar form in \cite{STV} in the particular case where $k$ is a field of characteristic $0$. In the case where $\sX$ is a classical scheme, the truncation $\tau_{\leq 1}(\sD_{\sX})$ is the classical determinant map as constructed in \cite{Knudsen-Mumford}.
\end{rem}

\begin{rem}
Notice that it is unclear if we can expect that the determinant from Theorem \ref{thm:determinant-of-perfect-complexes} can be generalized for (even connective) $\Einf$-rings. The main point is that for $\Einf$-algebras there are less functors which are left Kan extended from smooth $\Einf$-algebras, the reason being that this notion is not so well-behaved as in the case of simplicial commutative rings.

A piece of evidence that such an extension might not exist is Example 4.5 from \cite{Ausoni-Rognes}, where the authors show that it is not possible to construct a map from the K-theory of the sphere spectrum to its \emph{non-graded} Picard stack of line bundles.
\end{rem}

\subsection{Computation: The map induced on tangent spaces}

In this section we specialize to the case where $k$ is a field of characteristic $0$. See \cite{GR-I}*{Chapter 2} for the conventions on derived geometry for this set up. The reason that we restrict to this case is that we use the theory of deformation theory as developed in \cite{GR-II}*{Chapter 1} to compute the cotangent complexes involved.

Let
\[
f: \sX \ra \sY
\]
be a map between prestacks with deformation theory and locally almost of finite type, according to \cite{GR-II}*{Chapter 7}, for every point $x: S \ra \sX$ one obtains a map
\[
T_{x}(\sX)[-1] \ra T_{f(x)}(\sY)[-1]
\]
between dg Lie algebras. In this section we will describe more explicitly what this map of Lie algebras is for (\ref{eq:the-determinant-map}).

Let's recall the following construction. For any $S \in \Schaff$ and $\sG \in \QCoh(S)^{\leq 0}$ we can consider the square-zero extension of $S$ defined as $S_{\sG} = \Spec(\Gamma(S,\sO_S\oplus \sG))$. This comes equipped with maps
\[\imath_{\sG}: S \ra S_{\sG} \;\;\; \mbox{and}\;\;\; p_{\sG}: S_{\sG} \ra S\]
induced by the maps of sheaves of algebras $\sO_{S}\oplus \sG \ra \sO_S$ and $\sO_{S} \ra \sO_S\oplus \sG$, respectively. In the case where $\sG = \sO_{S}$ we will abbreviate those maps as $\imath$ and $p$.

The \emph{tangent complex} $T_{x}(\sVect)$ is characterized by the property that for any 
\[\sG \in \QCoh(S)^{-} = \cup_{k \geq 0}\QCoh(S)^{\leq k}\] 
one has
\begin{equation}
\label{eq:definition-tangent-complex}
\Hom_{\QCoh(S)^{-}}(\sG,T_{x}(\sVect)) \simeq \Omega^k\Maps_{S/}(S_{\sG}[k],\sVect),
\end{equation}
for some $k$ such that $\sG[k] \in \QCoh(S)^{\leq 0}$. The subscript on the mapping space means we consider maps $y$ that fit into the diagram
\[
\begin{tikzcd}
S_{\sG} \ar[rd,dashed,"y"] & \\
S \ar[u,"\imath_{\sG}"] \ar[r,"x"] & \sVect
\end{tikzcd}
\]

In the particular case where $\sG = \sO_{S}$, let $\sF_x \in \sVect(S)$ denote the vector bundle corresponding to $x: S \ra \sVect$ we can understand the right-hand side of (\ref{eq:definition-tangent-complex}) as
\begin{equation}
\label{eq:extension-of-modules}
\{\sF_y \in \QCoh(S_{\sO_S}) \; | \; \imath^*(\sF_y) \simeq \sF_x\}.
\end{equation}   

\begin{lem}
\label{lem:tangent-complex-computation}
The tangent complex of $\sVect$ is
\[T_{x}(\sVect) \simeq \Hom_{S}(\sF_{x},\sF_{x})[1] \simeq \{\sF_y \in \QCoh(S_{\sO_S}) \; | \; \imath^*(\sF_y) \simeq \sF_x\}.\]
\end{lem}

\begin{proof}
We prove the right isomorphism. The left isomorphism follows from the definition of the tangent complex.

Let's identify $\sO_{\sO_S} \simeq \sO_S[\epsilon]$, where $\epsilon^2 = 0$. Let $\sF_{y}$ be a $\sO_{S}[\epsilon]$-module such that $\imath^*\sF_y \simeq \sF_x$. Consider the fiber sequence
\[p^*(\sO_S) \ra \sO_{S}[\epsilon] \ra \imath_{*}\sO_S\]
of $\sO_S[\epsilon]$-modules. Tensoring it with $\sF_{y}$ yields
\[p^*(\sF_x) \ra \sF_y \ra \imath_{*}(\sF_{x}).\]
Hence, $\sF_y$ is determined by the map
\[\imath_{*}(\sF_{x}) \ra p^*(\sF_x)[1]\]
of $\sO_{S}[\epsilon]$-modules. Since
\[\Hom_{S_{\sO_S}}(p^*(\sF_x),\imath_*(\sF_x\otimes \sO_S[1])) \simeq \Hom_{S}(\sF_x,p_*\circ \imath_*(\sF_x\otimes \sO_{S}[1])) \simeq \Hom_S(\sF_x,\sF\otimes\sO_{S}[1]).\]

This finishes the proof.
\end{proof}

Let's analyze more concretely what the correspondence of Lemma \ref{lem:tangent-complex-computation} gives us.

Suppose one has $\beta \in \pi_0\Hom_{S}(\sO_S,T_{x}\sVect)$, this corresponds to an element
\[\beta \in \Ext^{1}_{S}(\sF_x,\sF_x).\]
Since $\sF_x$ is a vector bundle this Ext-group vanishes (\cite{HA}*{Proposition. 7.2.2.6. (3)}). In other words, up to a contractible space of choice there is only one $\sO_{S}[\epsilon]$-module $\sF_y$ whose restriction to $S$ recovers $\sF_x$.

However, if one considers $\alpha \in \pi_0\Hom_S(\sO_S[1],T_{x}\sVect)$, this gives an element in
\[\alpha \in \Ext^0(\sF_x,\sF_x),\]
i.e. an actual endomorphism of $\sF_x$. By (\ref{eq:definition-tangent-complex}) one obtains that $\alpha$ corresponds to an automorphism of $y$ over $x$, i.e. a map $\varphi$ between the fiber sequences
\begin{equation}
    \label{eq:fiber-sequence-square-zero-extension}
    \begin{tikzcd}
    \sF_x \ar[r] \ar[d,"\id_{\sF_x}"]  & \sF_y \ar[r] \ar[d,"\varphi"] & \sF_x \ar[d,"\id_{\sF_x}"] \\
    \sF_x \ar[r] & \sF_{y} \ar[r] & \sF_x
    \end{tikzcd}
\end{equation}

To determine $\varphi$ we fix an isomorphism
\[
\sF_{y} \simeq \sF_x[\epsilon]
\]
as $\sO_S$-modules. It is easy to see that
\[
\varphi(f + g \epsilon) = f + (g + \alpha(f))\epsilon,
\]
where $f,g$ are local section of $\sF_x$.

Recall that for $\sE \in \sPerf(S)$ a perfect complex, one has a canonical trace morphism
\[\tr: \Hom_S(\sE,\sE) \ra \sO_S.\]
Indeed, $\sE$ has a dual $\sE^{\vee}$ and given any map $f \in \Hom_S(\sE,\sE)$ one defines $\tr(f)$ to be the composite
\[\sO_S \overset{\coev_{\sE}}{\ra} \sE \otimes \sE^{\vee} \simeq \sE^{\vee} \otimes \sE \overset{\id_{\sE^{\vee}}\otimes f}{\ra} \sE^{\vee}\otimes\sE \overset{\ev_{\sE}}{\ra} \sO_S.\]

The following result has its proof sketched in \cite{STV}.

\begin{prop}
The map induced on tangent spaces by $\sD_S$
\[
\sD_{S,*}: \Hom_S(\sF_x,\sF_{x})[1] \ra \Hom_{S}(\bwedge{\rank(\sF_x)}\sF_x,\bwedge{\rank(\sF_x)}\sF_x)[1] \simeq \sO_{S}[1]
\]
is the trace morphism on perfect complexes.
\end{prop}

\begin{proof}
We notice that the discussion about the tangent space applies to the stack $\sPerf$. Given any $\sF_x \in \sPerf(S)$ and $\alpha \in \Hom_S(\sF_x,\sF_{x})[1]$ we have the diagram (\ref{eq:fiber-sequence-square-zero-extension}) associated to $\varphi = \id_{\sF_y} + \epsilon\cdot\alpha$. If we apply the map $\sD$ to this diagram one obtains
\[
\begin{tikzcd}
    \bwedge{\rank(\sF_x)}\sF_x \ar[r] \ar[d,"\id_{\sF_x}"]  & \bwedge{\rank(\sF_y)}\sF_y \ar[r] \ar[d,"\bwedge{\rank(\sF_y)}\varphi"] & \bwedge{\rank(\sF_x)}\sF_x \ar[d,"\id_{\sF_x}"] \\
    \bwedge{\rank(\sF_x)}\sF_x \ar[r] & \bwedge{\rank(\sF_y)}\sF_{y} \ar[r] & \bwedge{\rank(\sF_x)}\sF_x
\end{tikzcd}
\]
Thus, the map $\gamma = \sD_{S,*}(\alpha) \in \Hom_{S}(\bwedge{\rank(\sF_x)}\sF_x,\bwedge{\rank(\sF_x)}\sF_x)[1]$ is determined by
\[
\id_{\bwedge{\rank(\sF_y)}\sF_y} + \gamma\cdot\epsilon = \bwedge{\rank(\sF_y)}(\id_{\sF_y} + \alpha\cdot\epsilon) = \bwedge{\rank(\sF_y)}\varphi.
\]
Since $\epsilon^2 = 0$, one obtains that
\[\gamma = \frac{d}{d\epsilon}\left(\det(\id_{\sF_y} + \alpha\cdot\epsilon) - \id_{\bwedge{\rank(\sF_y)}\sF_y}\right) = \tr(\alpha).\]

This finishes the proof.
\end{proof}

\section{Higher determinant maps}
\label{sec:higher-determinant}

In this section we will bootstrap the determinant map of perfect complexes to a map from the prestack of Tate objects. In section \ref{subsec:recollection-on-Tate} we recall the definition of Tate objects, in section \ref{subsec:relative-S-construction} we introduce the construction that allows one to deloop the determinant map. The last two sections \ref{subsec:higher-determinant} and \ref{subsec:central-extensions} construct the higher determinant map and apply it to obtain central extensions of loop groups.

\subsection{Recollection on Tate objects}
\label{subsec:recollection-on-Tate}

The notion of Tate objects for $\infty$-categories was first introduced in \cite{Hennion-Tate}*{\S 2}. Here we take a slightly different approach where we keep track of cardinalities to avoid having to consider a bigger universe, which is closer to the approach of \cite{BGW-Tate}, where Tate objects are defined for ordinary exact categories, and was suggested in \cite{AGH}*{Remark 2.34}.

Given $\lambda$ a regular infinite cardinal we denote by $\Catl$ the category of essentially $\lambda$-small categories\footnote{See \cite{HTT}*{\S 1.2.15} for a discussion of this and \cite{HTT}*{Proposition 5.4.1.2} for a definition.} and $\Catstl$ the subcategory of stable categories and exact functors between these.

For the rest of the paper we fix $\kappa < \lambda$ regular infinite cardinals.

\begin{defn}
For any $\sC \in \Catl$ one has
\begin{itemize}
    \item (Ind-objects in $\sC$) the category $\Indk(\sC)$ is the subcategory of $\Fun(\sC^{\rm op},\Spc)$ generated by those functors which commute with $\kappa$-small colimits;
    \item (Pro-objects in $\sC$) the category $\Prok(\sC)$ is the subcategory of $\Fun(\sC,\Spc)^{\rm op}$ generated by those functors which commute with $\kappa$-small limits;
    \item (elementary Tate-objects in $\sC$) the category $\Tateelk(\sC)$ is the smallest subcategory of $\Indk(\Prok(\sC))$ containing the essential images of  $\Prok(\sC)$ and $\Indk(\sC)$ and closed under finite extensions;
    \item (Tate-objects in $\sC$) the category $\Tatek(\sC)$ is the idempotent completion of $\Tateelk(\sC)$.
\end{itemize}
\end{defn}

The reason we consider the constructions with respect to a smaller cardinal is to obtain the following result.

\begin{lem}
\label{lem:smallness-of-Tate-category}
The categories $\Indk(\sC)$, $\Prok(\sC)$, $\Tateelk(\sC)$ and $\Tatek(\sC)$ are essentially $\lambda$-small, i.e. equivalent to some category which is $\lambda$-small.
\end{lem}

\begin{proof}

For $\Indk(\sC)$ this is an immediate consequence of \cite{HTT}*{Proposition 5.3.5.12}, i.e.\ $\Indk(\sC)$ is equivalent to the $\kappa$-compact objects in $\sP(\sC) = \Fun(\sC^{\rm op},\Spc)$ and the latter category is essentially $\lambda$-small by \cite{HTT}*{Example 5.4.1.8}.

The result for $\Prok(\sC)$ follows by passing to opposite categories and for $\Tateelk(\sC)$ and $\Tate(\sC)$ follows from the result for $\Indk(\Prok(\sC))$.
\end{proof}

Moreover, if one consider $\sC$ a $\lambda$-small stable category, then one has

\begin{lem}
For $\sC \in \Catstl$, the categories $\Indk(\sC)$, $\Prok(\sC)$ and $\Tatek(\sC)$ are all $\lambda$-small stable categories.
\end{lem}

\begin{proof}
The proof follows from \cite{Hennion-Tate}*{Corollary 2.7}.
\end{proof}

Notice that because of Lemma \ref{lem:smallness-of-Tate-category} we can iterate the construction above. 

\begin{defn}
For any $n \in \bN$, let $\Tatek^n(\sC)$ be the $n$th iteration of the Tate construction applied to $\sC$.
\end{defn}

\begin{rem}
From now we will drop $\kappa$ from the notation of $\Ind$, $\Pro$ and $\Tate$. As we won't consider any other cardinal this should not cause any confusion.
\end{rem}

\subsection{Relative $S_{\bullet}$--construction}
\label{subsec:relative-S-construction}

The idea for the construction in this section is originally due to Brauling, Groechenig and Wolfson in \cite{BGW-Index}*{\S 3.1}. We essentially mimic their construction adapting it to $\infty$-categories.

Let $n \geq 1$, we have two inclusion maps
\[
s_{\rm R}: \Ar_{n-1} \ra \Ar_{n} \;\;\; \mbox{and}\;\;\; s_{\rm C}: \Ar_{n-1} \ra \Ar_n
\]
given by $s_{\rm R}(i,j) = (i+1,j)$ and $s_{\rm C}(i,j)=(i,j)$, i.e. $s_{\rm R}$ adds a row to the arrow diagram and $s_{\rm C}$ adds a column. We denote by $s^*_{\rm R},s^{*}_{\rm C}: \sS_{n}\sC \ra \sS_{n-1}\sC$ the corresponding restriction functors.

\begin{defn}
\label{defn:relative-S-construction}
For every $n \geq 0$, we define $\sGr_{n}\sC$ as the following pullback
\[
\begin{tikzcd}
\sGr_{n}\sC \ar[r] \ar[d] & \sS_{n+1}\Pro(\sC) \times \sS_{n+1}\Ind(\sC) \ar[d] \\
\sS_{n+2}\Ind(\Pro(\sC)) \ar[r,"s^{*}_{\rm R} \times s^{*}_{\rm C}"] & \sS_{n+1}\Ind(\Pro(\sC)) \times \sS_{n+1}\Ind(\Pro(\sC))
\end{tikzcd}
\]

It is equipped with canonical maps $\sGr_{n}\sC \overset{e_n}{\ra} \Tateel(\sC)$ and $\sGr_{n}\sC \overset{s^{*}_{\rm R}\circ s^{*}_{\rm C}}{\ra} \sS_{n}\sC$. The first map is explicitly given by evaluation at $(n,n) \in \Ar_n$, and the second map is given by forgetting the last column, and then further forgetting the first row.
\end{defn}

\begin{rem}
For $n=0$, $\sGr_0\sC$ consists of diagrams
\[
\begin{tikzcd}
    X_{(0,0)} \ar[r] & X_{(0,1)} \ar[r] \ar[d] & X_{(0,2)} \ar[d] \\
    & X_{(1,1)} \ar[r] & X_{(1,2)} \ar[d] \\
    & & X_{(2,2)}
\end{tikzcd}
\]
in $\Ind(\Pro(\sC))$ where the subdiagram
\[
\begin{tikzcd}
    X_{(0,0)} \ar[r] & X_{(0,1)} \ar[d] \\
    & X_{(1,1)}
\end{tikzcd}
\]
belongs to $\Ind(\sC)$ and the subdiagram
\[
\begin{tikzcd}
    X_{(1,1)} \ar[r] & X_{(1,2)} \ar[d] \\
    & X_{(2,2)}
\end{tikzcd}
\]
belongs to $\Pro(\sC)$. 

Since $X_{(0,0)} \simeq X_{(1,1)} \simeq X_{(2,2)} \simeq 0$ by definition, we have that $\sGr_{0}\sC$ is equivalent to pushout squares
\[
\begin{tikzcd}
    X_{(0,1)} \ar[r] \ar[d] & X_{(0,2)} \ar[d] \\
    0 \ar[r] & X_{(1,2)}
\end{tikzcd}
\]
in $\Ind(\Pro(\sC))$, where $X_{(0,1)} \in \Ind(\sC)$ and $X_{(1,2)} \in \Pro(\sC)$. Notice that by definition $X_{(0,2)}$ belongs to $\Tateel(\sC)$.
\end{rem}

The following observation is behind the construction of the index map.

\begin{lem}
\label{lem:Grassmannian-factors}
For any $V \in \Tateel(\sC)$ let $\sGr_n(V) = e^{-1}_n(V)$. One has
\[\sGr_n(V) \simeq \Fun(\Delta^n,\sGr_0(V)).\]
\end{lem}

\begin{proof}
Indeed, by \cite{HA}*{Lemma 1.2.2.4} one has an isomorphism $\sS_{n}\sC \simeq \Fun(\Delta^{n-1},\sC)$ functorial in $\sC$. Applied to the pullback diagram defining $\sGr_n\sC$ we obtain
\[
\begin{tikzcd}
\sGr_{n}\sC \ar[r] \ar[d] & \Fun(\Delta^n,\Pro(\sC)) \times \Fun(\Delta^{n+1},\Ind(\sC)) \ar[d] \\
\Fun(\Delta^{n+1},\Ind(\Pro(\sC))) \ar[r,"s^*_{\rm R} \times s^{*}_{\rm C}"] & \Fun(\Delta^n,\Ind(\Pro(\sC)))^{\times 2}
\end{tikzcd}
\]
Since $\Fun(\Delta^{n+1},\Ind(\Pro(\sC))) \simeq \Fun(\Delta^n,\Fun(\Delta^1,\sC))$, and we can commute $\Fun(\Delta^n,-)$ with the pullback to obtain the equivalence
\[
\sGr_n\sC \simeq \Fun(\Delta^n,\sGr_0(\sC)).
\]
The result for $V$ follows from taking the fibers.
\end{proof}

By the naturality of the construction the categories $\sGr_n\sC$ assemble into a simplicial object. We let $\Gr_{\bullet}(\sC)$ denote the simplicial space obtained by passing to the underlying $\infty$-groupoid levelwise and $e_{\bullet}: \Gr_{\bullet}(\sc) \ra \Tateel(\sC)^{\simeq}$ the map to the underlying $\infty$-groupoid of the category of elementary Tate objects seen as a constant simplicial object in spaces.

\begin{cor}
\label{cor:invertible-projection}
The map induced by $e_{\bullet}$ upon geometric realization:
\begin{equation}
\label{eq:contractible-equivalence}
e: \left|\Gr_{\bullet}\sC\right| \ra \Tateel(\sC)^{\simeq}
\end{equation}
is an equivalence.
\end{cor}

\begin{proof}
For any $V \in \Tateel(\sC)^{\simeq}$ the fiber of $e$ is the geometric realization of $\sGr_{\bullet}(V)^{\simeq}$.

The latter is levelwise isomorphic to $\Fun(\Delta^n,\sGr_0(V))$ by Lemma \ref{lem:Grassmannian-factors}. By \cite{Hennion-Tate}*{Theorem 3.15} one has that $\sGr_0(V)$ is cofiltered, which implies that for each $[n] \in \Delta^{\rm op}$ the space $\Gr_n(V)$ is contractible, and so is the geometric realization of $\Gr_{\bullet}(\sC)$.
\end{proof}

Corollary \ref{cor:invertible-projection} allows us to define the map
\begin{equation}
\label{eq:zeroth-delooping-map}
\widetilde{\sD}: \Tateel(\sC)^{\simeq} \overset{\simeq}{\la} \left|\Gr_{\bullet}(\sC)\right| \overset{s^*_{\rm R}\circ s^{*}_{\rm C}}{\ra} \left|S_{\bullet}(\sC)\right|,    
\end{equation}
where we keep the same names $s^*_{\rm R}$ and $s^*_{\rm C}$ for the maps induced upon the geometric realization of the underlying simplicial spaces.

\begin{rem}
\label{rem:corollary-is-general}
For any $k\geq 1$ we can consider $\sS_{\bullet_k}(\sC)$ instead of $\sC$ in the statements of Lemma \ref{lem:Grassmannian-factors} and Corollary \ref{cor:invertible-projection} and they can be proved in the same way. Moreover, the corresponding equivalence maps
\[
e_k: \left|\Gr_{\bullet}\sS_{\bullet_k}(\sC)\right| \overset{\simeq}{\ra} \sS_{\bullet_k}(\Tateel(\sC))^{\simeq}
\]
are compatible with the suspension maps exhibiting the structure of a spectrum on
\[
\left(\left|S_{\bullet}(\Tateel(\sC))\right|,\left|S_{\bullet_2}(\Tateel(\sC))\right|,\ldots\right).
\]

\end{rem}

Similarly to (\ref{eq:zeroth-delooping-map}) by Remark \ref{rem:corollary-is-general} we obtain maps
\begin{equation}
\label{eq:kth-delooping-map}
\widetilde{\sD}_{k}: \left|S_{\bullet_k}(\Tateel(\sC))\right| \overset{\simeq}{\la} \left|\sGr_{\bullet}(\sS_{\bullet_k}(\sC))\right| \ra \left|S_{\bullet_{k+1}}\sC\right|
\end{equation}
for each $k\geq 1$.

Those assemble to give
\begin{cor}
There is a map of spectra
\[
\sD_{\sC}: \K(\Tate(\sC)) \ra \B\K(\sC),
\]
where $\B\K(\sC)$ is the suspension of the spectrum $\K(\sC)$.
\end{cor}

\begin{rem}
One notices that $\K(\Tate(\sC)) \simeq \K(\Tateel(\sC))$ because of Morita invariance of K-theory (see \cite{BGT}*{} or \cite{Barwick}*{}). Another reason why one obtains $\B\K(\sC)$ as the target of $\sD_{\sC}$ above is that the target spectrum is explicitly given by the sequence of spaces
\[
(\left|S_{\bullet_2}\sC\right|,\left|S_{\bullet_3}\sC\right|,\ldots),
\]
and one has equivalences $\Sigma\left|S_{\bullet_k}\sC\right| \simeq \left|S_{\bullet_{k+1}}\sC\right|$ for all $k \geq 1$, by Remark \ref{rem:S-construction-is-delooping}.
\end{rem}

\subsection{Determinant of Tate objects}
\label{subsec:higher-determinant}

For any integer $n\geq 1$ and derived ring $R$, i.e.\ connective dg $k$-algebra we can inductively define the category 
\[
\Tate^{n}(R) = \Tate(\Tate^{n-1}(\Perf(R))),
\]
where $\Tate^0(\Perf(R)) = \Perf(R)$. 

In particular, for any $n\geq 1$ one obtain prestacks
\begin{align*}
    \sTate^{(n)}: \Schaffop & \ra \Spc \\
    S & \mapsto \Tate^{n}(\Perf(S))^{\simeq}.
\end{align*}

Similarly, one can also consider the prestacks
\begin{align*}
    \sBPicgrn{n}: \Schaffop & \ra \Spc \\
    S & \mapsto \B\sBPicgrn{n-1},
\end{align*}

where $\sBPicgrn{0} = \sPicgr$, and for any group object $\sA$ in prestacks one has
\[
\B\sA = \colim_{\Delta^{[n] \in \rm op}}\sA^{\times n},
\]
with simplicial maps induced by the group structure.

The construction from the previous section can be iterated to give
\begin{equation}
\label{eq:determinant-n-K-theory-categories}
\sD^{(n)}_{\Perf(S)}: \K(\Tate^{n}(\Perf(S))) \ra \B^{n}\K(\Perf(S))
\end{equation}
for any $S \in \Schaffop$.

In particular, since there are maps
\[
\imath: \sTate^{n}(S) \ra \K(\Tate^{n}(\Perf(S))) \;\;\; \mbox{and} \;\;\; \tau^{\leq (n+2)}: \B^{n}\K(\Perf(S)) \ra \sBPicgrn{n}(S),
\]
one obtains the following

\begin{cor}
\label{cor:higher-determinant-for-prestacks}
For any $n\geq 1$ there exists a map of prestacks
\begin{equation}
\label{eq:determinant-n-prestacks}
\sD^{(n)}: \sTate^{(n)} \ra \sBPicgrn{n}.    
\end{equation}
\end{cor}

\begin{rem}
Since the map (\ref{eq:determinant-n-prestacks}) is obtained as the restriction of a map (\ref{eq:determinant-n-K-theory-categories}) of K-theory spectra, one can informally say that (\ref{eq:determinant-n-prestacks}) can be enhanced to witness the multiplicative property of the determinant. More precisely, if one thinks of the category of $n$-Tate objects in perfect complexes as a (coCartesian) symmetric monoidal category with respect to \emph{direct sum}, then the functor $\sD^{(n)}$ has a symmetric monoidal structure. 
\end{rem}

Suppose that $\sX$ is any prestack, one can consider the moduli of $n$-Tate objects over $\sX$ as follows:\footnote{Notice that here $\Maps_{\PStk}$ denotes simply taking hom as prestacks and not the mapping stack, i.e.\ there is not stackification after.}
\[
\sTate^{(n)}(\sX) = \Maps_{\PStk}(\sX,\sTate^{(n)}).
\]

By composing with (\ref{eq:determinant-n-prestacks}) one obtains a map
\begin{equation}
    \label{eq:higher-determinant-for-sX}
    \sD^{(n)}_{\sX}: \sTate^{(n)}(\sX) \ra  \sBPicgrn{n}(\sX),
\end{equation}
where $\sBPicgrn{n}(\sX)$ is the space of $n$-$\sPicgr$-gerbes over $\sX$.

\begin{rem}
In \cite{STV} they used the map (\ref{eq:determinant-for-sX}) for $n=0$ and $\sX$ a K3 surface to prove that the (derived) moduli of simple perfect complexes with non-negative self-Ext's and fixed determinant is smooth. We reiterate their remark that we expect the map (\ref{eq:higher-determinant-for-sX}) to be useful in moduli space problems which can be related to the moduli space of Tate objects.
\end{rem}

\begin{rem}
More generally, if one wants to formulate the property that the determinant also splits non-split fiber sequences of $n$-Tate objects in perfect complexes then one needs to formulate the map $\sD^{(n)}$ as the $0$th level map of certain (higher) Segal objects in the category of prestacks. We plan to pursue this characterization in some future work.
\end{rem}

\subsection{Application: central extension of iterated loop groups}
\label{subsec:central-extensions}

We can now reap the fruits of our higher determinant map, in this section we take $k$ to be an arbitrary commutative ring. Given $V \in \Tate^n(k)$ one considers the group prestack
\[
\sGL_V(S) = \Aut_{\Tate^n(S)}(V\otimes_k \sO_S).
\]
In particular, we let $\sLGL{n} = \sGL_{k((t_1))\cdots ((t_n))}$ denote the group prestack whose $R$-points are
\[
\sLGL{n}(R) = \Aut_{\Tate^n(R)}(R((t_1))\cdots((t_n))).
\]

The determinant map (\ref{eq:determinant-n-prestacks}) gives a map
\begin{equation}
\label{eq:line-bundle-on-loop-stack}
\sD^{(n)}(S)_*: \sLGL{n}(S) \ra \Aut_{\sBPicgrn{n}}(\sD^{(n)}(S)(\sO_{S}((t_1))\cdots ((t_n))) \simeq \sBPicgrn{n-1}(S).
\end{equation}

Indeed, one notices that for any $\sG \in \sBPicgrn{n}(S)$ one has a canonical equivalence
\[
\Aut_{\sBPicgrn{n}}(\sG) \simeq S \times_{\sBPicgrn{n}} S \simeq \sBPicgrn{n-1}(S).
\]

\begin{lem}
The map (\ref{eq:line-bundle-on-loop-stack}) is a map of group prestacks.
\end{lem}

\begin{proof}
This follows automatically from the fact that $\sD^{(n)}$ is a map of prestacks.
\end{proof}

\begin{defn}
\label{defn:n-central-extension-prestack}
The central extension of $\sLGL{n}$ is defined as
\[
\sLGL{n}^{\wedge} = \Fib(\sD^{(n)}(S)_*),
\]
where the fiber is taken in the category of group objects in prestacks.
\end{defn}

\begin{rem}
\label{rem:center-of-central-extension}
For $V = k((t_1))\ldots((t_n))$, the map $\sD^{(n)}(V): \Spec(k) \ra \sBPicgrn{n}$, determines
\[
\left.\sD^{(n)}\right|_{V} \ra \Spec(k)
\]
a $\sBPicgrn{n-1}$-torsor over $\Spec(k)$. The fibers of $\sLGL{n}^{\wedge} \ra \sLGL{n}$ are canonically equivalent to $\Aut(\left.\sD^{(n)}\right|_{V})$.
\end{rem}

We now compare the central extension $\sLG^{\wedge} = \sLGL{1}^{\wedge}$ with previous constructions in the literature. We need some notation first.

Let $\classical{\sTate}$ denote the restriction of $\sTate$ to the category of classical schemes $(\Schaffcl)^{\rm op}$, namely
\[\classical{\sTate}(R) = \Tate(\Perf(\pi_0(R)))^{\simeq}.\]

Notice that for $\pi_0(R)$ an ordinary commutative ring, one can consider $\oTate(\pi_0(R))$ the ordinary exact category of Tate objects over $\oModfgp_{\pi_0(R)}$, the exact category of finitely generated projective modules over $\pi_0(R)$. We claim that there is a map\footnote{Here $\h\sC$ denotes the ordinary homotopy category underlying the ($\infty$-)category $\sC$.}
\begin{equation}
\label{eq:Tate-exact-to-infty}
    \oTate(\pi_0(R))^{\simeq} \ra \h\,\classical{\Tate}(R)^{\simeq}.
\end{equation}
Indeed, for $A$ an ordinary commutative ring the canonical inclusion of\footnote{Recall that the category of perfect modules over a commutativer ring $A$ can be concretely described as compact objects in the category $\N_{\rm dg}(\Ch(A))$, i.e.\ the dg-nerve (see \cite{HA}*{Construction 1.3.1.6}) of the differential graded ordinary category of chains complexes of $A$-modules.} 
\[\oMod_{A} \ra \Perf(A)\]
induces a map between the associated Tate constructions
\[\oTate(\oModfgp_A) \ra \Tate(\Perf(A)).\]
By passing to the underlying ordinary homotopy category and underlying ordinary groupoid one obtains the map (\ref{eq:Tate-exact-to-infty}).

Let $\classical{\sD}: \classical{\sTate} \ra \classical{\sBPicgr}$ be the map of classical prestacks obtained by applying $\classical{(-)}$ to (\ref{eq:determinant-n-prestacks}) for $n=1$. Consider $\classical{\sD}^{\leq 1}$ the composite of $\classical{\sD}$ with $\tau^{\leq 1}$ the localization from spaces to $1$-truncated spaces.

\begin{lem}
\label{lem:set-of-trivializations}
The restriction of $\classical{\sD}^{\leq 1}(R)$ to $\oTate(\pi_0(R))^{\simeq}$ agrees with the map
\[\B(\mbox{det}^{\rm gr}) \circ \mbox{Index}: \oTate(\pi_0(R))^{\simeq} \ra B\mbox{Pic}^{\rm gr}_{\pi_0(R)}\]
of \cite{BGW-Index}*{Section 3}\footnote{In \cite{BGW-Index} they use the notation $\mbox{Pic}^{\bZ}$ for $\mbox{Pic}^{\rm gr}$, and similarly for the determinant map.}.

In particular, the space of sections $\Gamma(\Spec(\pi_0(R)),\left.\classical{\sD}(R)\right|_{\{[V]\}})$ is as described in Proposition 5.1 of \cite{BGW-Index}.
\end{lem}

\begin{proof}
We first notice that by \cite{SAG}*{Proposition 2.9.6.2}, one has an equivalence
\[
\tau^{\leq 1}\sBPicgr \simeq B\tau^{\leq 0}(\sPicgr) \simeq B\mbox{Pic}^{\rm gr}_{\pi_0(R)}.
\]

The compatibility of the maps follows from the functoriality of the construction of the index map in both cases.
\end{proof}

Let $\LG$ denote the functor
\[
\tau^{\leq 0}\circ \classical{\sLG}: \Schaffclop \ra \Spc^{\leq 0} \simeq \Set,
\]
where $\tau^{\leq 0}$ is the Postnikov truncation of spaces to $0$-trucanted spaces, i.e. discrete sets. By applying the same functors to $\sLG^{\wedge}$ we obtain a central extension
\[
1 \ra \Gm \ra \widehat{\LG} \ra \LG \ra 1,
\]
i.e.\ $\widehat{\LG} \simeq \tau^{\leq 0}\classical{\sLG^{\wedge}}$.

\begin{prop}
\label{prop:1-dim-central-agrees}
The central extension $\widehat{\LG}$ agrees with the central extension denoted by $\widehat{\GL}'_{\infty}$ in \cite{FZ}*{\S 1.2.3}.
\end{prop}

\begin{proof}
By Remark \ref{rem:center-of-central-extension} one notices that 
\[\Fib(\widehat{\LG} \ra \LG) \simeq \Aut(\left.\classical{\sD}^{\leq 1}\right|_{k((t))}).\]
Hence it is enough to identify the later with the description in Frenkel-Zhu. 

From Lemma \ref{lem:set-of-trivializations} one has that $\Aut(\left.\classical{\sD}^{\leq 1}\right|_{V})$ is non-canonically isomorphic to automorphism of the determinantal theory that sends $L_0 = k[[t]]$ to $k$ and any other lattice $L \subset k((t))$ to
\[
\bwedge{\rank((L+k[[t]])/L\cap k[[t]])}((L+k[[t]])/L\cap k[[t]]).
\]
This is clearly equivalent to the set of automorphism the determinant line bundle over the affine Grassmannian $\Gr$ at $L_0$, as one can see from the proof of Proposition 1.6 from \cite{FZ}.
\end{proof}

\begin{rem}
When $k$ is a field of characteristic $0$ we expect the restriction of $\sLGL{2}^{\wedge}$ to classical prestacks recover the central extension of \cite{FZ}*{Section 3.3} denoted $\mathbf{G}\mathbf{L}_{\infty,\infty}$ in \emph{loc. cit.}. 
\end{rem}

\begin{rem}
More generally, for $k$ an arbitrary commutative ring we also expect that the restriction of $\sLGL{2}^{\wedge}$ to classical prestacks to recover the central extension defined in \S 5.2 of \cite{Osipov-Zhu}.
\end{rem}

\appendix

\section{Proofs for Section \ref{sec:K-theory}.}

In this section we prove some of the main results about K-theory that we summarized in \S \ref{sec:K-theory}. Our goal is to present some of the arguments in a slicker form than what can be found in the literature, though all of those results were already in \cite{Barwick}, \cite{Lurie-K-theory} and \cite{Fontes}.

\subsection{Realization fibrations}
\label{subsec:realization-fibration}

In this section we introduce a preliminary technical tool, developed by Rezk in \cite{Rezk}, which allows for a simple proof of the Additivity theorem in the next section.

\begin{defn}
Let $I$ be a small category and $f:X \ra Y$ a morphism in $\Fun(I^{\rm op},\Spc)$. We say that $f$ is a \emph{realization fibration} if for every pullback
\[
\begin{tikzcd}
X' \ar[r] \ar[d] & X \ar[d,"f"] \\
Y' \ar[r] & Y
\end{tikzcd}
\]
the diagram obtained by taking the colimit over $I$
\[
\begin{tikzcd}
\colim_I X' \ar[r] \ar[d] & \colim_I X \ar[d,"f"] \\
\colim_I Y' \ar[r] & \colim_I Y
\end{tikzcd}
\]
is a pullback diagram in $\Spc$.
\end{defn}

\begin{rem}
\label{rem:properties-of-RF}
Notice that any map $f$ which induces an equivalence upon passing to colimits is a realization fibration diagram, and that the class of such is stable under pullbacks.
\end{rem}

In \cite{Rezk}, the following definition is introduced as part of the local-to-global principle to check that a map is a realization fibration.

\begin{defn}
\label{defn:equifibered}
Let $J$ be a small category, $F,G:J \ra \sC$  be functors and $p:F \ra G$ a natural transformation, one says that $p$ is \emph{$J$-equifibered} if for every morphism $f:j_2 \ra j_1$ in $J$ the diagram
\[
\begin{tikzcd}
F(j_2) \ar[r,"F(f)"] \ar[d,"p"] & F(j_1) \ar[d,"p"] \\
G(j_2) \ar[r,"G(f)"] & G(j_1) 
\end{tikzcd}
\]
is a pullback square.
\end{defn}

Rezk then proves the following useful criterion\footnote{Notice that in \cite{Rezk} this is phrased in the language of model categories, but the argument goes also works for the $\infty$-category that we consider here.}.

\begin{prop}[\cite{Rezk}*{Theorem 2.6}]
Let $J$ be a small category, $F,G: J \ra \sP(I)$\footnote{Here $\sP(I)$ is the category of space-valued presheaves on a small category $I$.} functors and $p:F \ra G$ a natural transformation, suppose that $p$ is a $J$-equifibered map, and that for each $j \in J$ the map $p(j):F(j) \ra G(j)$ is a realization fibration. Then
\[
\colim_J F(j) \ra \colim_J G(j)
\]
is a realization fibration.
\end{prop}

In the particular case when $I = \Delta$ Rezk proves the following result, which can be seen as a way to bypass the $\pi_*$-Kan condition.

\begin{lem}[\cite{Rezk}*{Proposition 5.4}]
\label{lem:Rezk-criterion-simplicial-spaces}
For $f: X \ra Y$ a map of simplicial spaces, if for all $[m] \in \Delta^{\rm op}$ the map
\[Y([m]) \ra Y([0])\]
induces an isomorphism on $\pi_0$, then $f$ is a realization fibration.
\end{lem}

An important consequence of Lemma \ref{lem:Rezk-criterion-simplicial-spaces} is the following

\begin{lem}
\label{lem:only-calculation}
For any $[m] \in \Delta^{\rm op}$ and any Waldhausen category $\sC$ the maps
\begin{enumerate}[(i)]
    \item 
    \[q^{(m)}: S_{\bullet+m+1}\sC \ra S_{m}\sC\]
    \item 
    \[p^{(m)}: S_{\bullet+m+1}\sC \ra S_{\bullet}\sC,\]
\end{enumerate}
where $q^{(m)}$ is the projection onto the last $(m-1)$ elements and $p^{(m)}$ is given by the projection onto the first elements of the simplicial set, induce equivalences upon geometric realization\footnote{We are considering $S_{m}\sC$ as the constant simplicial object in spaces.}. 
\end{lem}

\begin{proof}
We prove (i); the argument for (ii) is completely analogous. Since $\sC$ is pointed, $|S_{m}\sC|$ is contractible. Hence, it is enough to check that the map $p^{(m)}$ is a realization fibration. By Lemma \ref{lem:Rezk-criterion-simplicial-spaces} it is enough to check that it induces an equivalence on $\pi_0$, and that again follows from the fact that $S_{m+1}\sC$ is pointed.
\end{proof}

\subsection{Additivity Theorem}
\label{subsec:additivity-appendix}

Our exposition is heavily influenced by the paper \cite{McCarthy}. First we introduce some notation and prove a couple of preliminary results. For $X \in \Fun(N\Delta^{\rm op},\sC)$ a simplicial object in a category $\sC$, we denote by $X^R$ and $X^L$ the bisimplicial objects obtained as
\[X^L_{n,m} = X_n \;\;\; \mbox{and}\;\;\; X^R_{n,m} = X_m.\]

Suppose that $F: \sC \ra \sD$ is a functor between Waldhausen categories. Then we have an induced functor $\sS_{\bullet}F: \sS_{\bullet}\sC \ra \sS_{\bullet}\sD$ of simplicial objects in $\Wald$. We consider the bisimplicial object $\sS_{\bullet_{2}}\left.F\right|_{\sD}$ whose value on $([n],[m])$ is defined as the following pullback

\begin{equation}
\label{eq:definition-relative-double-S-construction}
    \begin{tikzcd}
    \sS_{n,m}\left.F\right|_{\sD} \ar[r] \ar[d] & \sS_{n}\sC\ar[d,"F_{n}"] \\
    \sS_{n+m+1}\sD\ar[r,"p^{(m)}_{n}"] & \sS_{n}\sD 
    \end{tikzcd}
\end{equation}
where $p^{(m)}_{n}$ is the projection onto the first $n+1$ elements of $\sS_{n+m+1}\sD$\footnote{I.e.\ the composition of the $m+1$ face maps $d^{k}$ for $n+1 \leq k \leq n+m+1$.}. As before we will denote by unscripted $S$ the corresponding objects obtained by passing to the underlying $\infty$-groupoid.

By construction we have maps of bisimplicial objects
\[\pi^{F}_{\bullet_2}: S_{\bullet_{2}}\left.F\right|_{\sD} \ra S^{\rm L}_{\bullet_2}\sC\]
and 
\[
\rho^{F}_{\bullet_2}: S_{\bullet_{2}}\left.F\right|_{\sD} \ra S^{\rm R}_{\bullet_2}\sD,
\]
where $\rho^F_{n,m}$ is the composite of the projection onto $S_{n+m+1}\sD$ with $d^{n+1}_0: S_{n+m+1}\sD \ra S_{m}\sD$.

\begin{prop}
\label{prop:trick-with-bisimplicial-construction}
The following are equivalent
\begin{enumerate}[a)]
    \item the map $S_{\bullet}F$ induces an equivalence upon geometric realization;
    \item the map $\rho^{F}_{\bullet_2}$ induces an equivalence upon geometric realization.
\end{enumerate}
\end{prop}

\begin{proof}
Consider the diagram
\begin{equation}
\label{eq:key-diagram-for-rho}
    \begin{tikzcd}
    S^{\rm R}_{\bullet_2}\sD \ar[d,"S_{\bullet_2}\id_{\sD}"] & S_{\bullet_2}\left.F\right|_{\sD} \ar[r,"\pi^{F}_{\bullet_2}"] \ar[l,"\rho^F_{\bullet_2}"] \ar[d,"S_{\bullet_2}F"] & S^{\rm L}_{\bullet_2}\sC \ar[d,"S^{\rm L}_{\bullet_2}F"] \\
    S^{\rm R}_{\bullet_2}\sD & S_{\bullet_2}\left.\id_{\sD}\right|_{\sD} \ar[r,"\pi^{\id_{\sD}}_{\bullet_2}"] \ar[l,"\rho^{\id_{\sD}}_{\bullet_2}"] & S^{\rm L}_{\bullet_2}\sD
    \end{tikzcd}
\end{equation}

We claim that (1) $S_{\bullet_2}\id_{\sD}$, (2) $\rho^{\id_{\sD}}_{\bullet_2}$, (3) $\pi^{\id_{\sD}}_{\bullet_2}$ and (4) $\pi^{F}_{\bullet_2}$ induce equivalences upon geometric realization.

The claim for (1) is clear, i.e. isomorphisms are realization fibrations.

For (2), notice that for each $[m] \in \Delta^{\rm op}$, $\rho^{\id_{\sD}}_{\bullet,m}$ is the composite
\[\left.S_{\bullet,m}F\right|_{\sD} \ra S_{\bullet+m+1}\sD \overset{q^{(m)}}{\ra} S_m\sD,\]
where the first map is induced by the pullback of $\id_{\sD,\bullet}$ in the defining square of $\left.\sS_{\bullet,m}F\right|_{\sD}$. Thus, the first map of the composition induces an equivalence upon geometric realization, by Lemma \ref{lem:only-calculation} (i), the second map also induces an equivalence upon geometric realization.

For (3) and (4), notice that for any functor $F:\sC \ra \sD$ and $[m] \in \Delta^{\rm op}$ one has
\[\Fib(\pi^F_{\bullet,m}) \simeq \Fib(p^{(m)}_{\bullet}),\]
where $p^{(m)}_{n}: \sS_{n+m+1}\sD \ra \sS_{n}\sD$. So the result follows again from Lemma \ref{lem:only-calculation} (ii).
\end{proof}

Here is essentially a reformulation of Proposition \ref{prop:trick-with-bisimplicial-construction} in the language of \S \ref{subsec:realization-fibration}.

\begin{prop}
\label{prop:RF-for-F-from-rho}
If the map of bi-simplicial objects $\rho^F_{\bullet_2}$ is a realization fibration, then so is $S_{\bullet}F$.
\end{prop}

\begin{proof}
Let $\sE_{\bullet}$ be a simplicial space and consider the pullback diagram
\[
\begin{tikzcd}
\sE' \ar[r] \ar[d] & S_{\bullet}\sC \ar[d,"S_{\bullet}F"] \\
\sE \ar[r] & S_{\bullet}\sD.
\end{tikzcd}
\]
This induces two pullback diagrams of bi-simplicial space by considering the functors $(-)^{L}$ and $(-)^R$ applied to everything in sight. Now consider the $\tilde{\sE}_{\bullet_2}$ and $\tilde{\sE'}_{\bullet_2}$ defined as the following pullbacks
\[
\begin{tikzcd}
\tilde{\sE}_{\bullet_2} \ar[r] \ar[d] & \sE^{\rm L}_{\bullet_2} \ar[d] \\
S_{\bullet_2}\left.\id_{\sD}\right|_{\sD} \ar[r,"\pi^{\id_{\sD}}_{\bullet_2}"] & S^{\rm L}_{\bullet_2}\sD
\end{tikzcd}
\;\;\;
\mbox{and}
\;\;\;
\begin{tikzcd}
\tilde{\sE'}_{\bullet_2} \ar[r] \ar[d] & \sE'_{\bullet_2} \ar[d] \\
\tilde{\sE}_{\bullet_2} \ar[r] & \sE_{\bullet_2}.
\end{tikzcd}
\]

Thus, we obtain the following diagram living over the diagram (\ref{eq:key-diagram-for-rho})
\begin{equation}
\label{eq:proof-RF-rho-implies-RF-f}
\begin{tikzcd}
S^{\rm R}_{\bullet_2}\sE' \ar[d,"s^{\rm R,\id_{\sD}}"] & \tilde{\sE'}_{\bullet_2} \ar[d,"\tilde{s}^{F}"] \ar[l,"r^F"] \ar[r,"p^F"] & \sE'_{\bullet_2} \ar[d,"s^{\rm L, F}"] \\
S^{\rm R}_{\bullet_2}\sE & \tilde{\sE}_{\bullet_2} \ar[l,"r^{\id_{\sD}}"] \ar[r,"p^{\id_{\sD}}"] & \sE_{\bullet_2}
\end{tikzcd}
\end{equation}
where the maps are the pullback of the corresponding maps in diagram (\ref{eq:key-diagram-for-rho}).

By considering the new six-term diagram
\[\Fib\left((\ref{eq:proof-RF-rho-implies-RF-f}) \ra (\ref{eq:key-diagram-for-rho})\right)\]
and passing to its geometric realization we see that if $\pi^F_{\bullet_2}$, $S_{\bullet_2}\id_{\sD}$, $\rho^{\id_{\sD}}_{\bullet_2}$, $\pi^{\id_{\sD}}_{\bullet_2}$ and $\rho^F_{\bullet_2}$ are realization fibrations, then so is $S_{\bullet_2}F$. However, by Proposition \ref{prop:trick-with-bisimplicial-construction} the maps $S_{\bullet_2}\id_{\sD}$, $\rho^{\id_{\sD}}_{\bullet_2}$, $\pi^{\id_{\sD}}_{\bullet_2}$ and $\pi^F_{\bullet_2}$ induce equivalences upon geometric realization. This finishes the proof.
\end{proof}

The final piece to prove the Additivity Theorem is the following.

\begin{prop}
\label{prop:any-rho-is-RF}
For any functor $F: \sC \ra \sD$ between Waldhausen categories, the induced map of bi-simplicial spaces
\[\rho^F_{\bullet_2}: S_{\bullet_2}\left.F\right|_{\sD} \ra S^R_{\bullet_2}\sD\]
is a realization fibration.
\end{prop}

\begin{proof}
Notice that the geometric realization of any bi-simplicial set $\sE_{\bullet_2}$ is equivalent to
\[\colim_{[m] \in \Delta^{\rm op}}\left|\sE_{\bullet,m}\right|,\]
where $\left|\sE_{\bullet,m}\right|$ is the geometric realization of $[n] \mapsto \sE_{n,m}$. Thus, we are in the situation to apply \cite{Rezk}*{Theorem 2.6}. We check its two conditions:
\begin{enumerate}[1)]
    \item for all $[m] \ra [m']$ in $\Delta$ the diagram
    \[
    \begin{tikzcd}
    \left.S_{\bullet,m'}F\right|_{\sD} \ar[r] \ar[d] & \left.S_{\bullet,m}F\right|_{\sD} \ar[d] \\
    S_{m'+1}\sD \ar[r] & S_{m+1}\sD
    \end{tikzcd}
    \]
    is a pullback diagram. This follows immediately from the definition of $\left.S_{\bullet_2}F\right|_{\sD}$, see diagram (\ref{eq:definition-relative-double-S-construction}).
    \item for all $[m] \in \Delta^{\rm op}$ the map
    \[\rho^F_{\bullet,m}: \left.S_{\bullet,m}\right|_{\sD} \ra S_{m+1}\sD\]
    is a realization fibration, where the right-hand side is the constant simplicial space $S_{m+1}\sD$. This follows from Lemma \ref{lem:Rezk-criterion-simplicial-spaces}.
\end{enumerate}
\end{proof}

Recall the set up of the Additivity Theorem, for $\sC$ a Waldhausen category we considered the functor
\[
F: \Fun([1],\sC) \ra \sC\times \sC
\]
given by evaluation on the source and taking the fiber of the morphism.

\begin{prop}
\label{prop:additivity-appendix}
For any Waldhausen category $\sC$ the functor $F$ induces an equivalence of K-theory spectra.
\end{prop}

\begin{proof}
From the construction of K-theory (see \S \ref{subsubsec:Waldhausen-categories}) it is enough to prove that the induced map
\begin{equation}
\label{eq:S-construction-of-morphisms}
S_{\bullet}\Fun([1],\sC) \ra S_{\bullet}\sC\times S_{\bullet}\sC    
\end{equation}
induces an equivalence upon passing to geometric realization. Equivalently, if one considers the fiber of (\ref{eq:S-construction-of-morphisms}) over the first factor of the right-hand side one has a fiber sequence
\begin{equation}
\label{eq:fiber-sequence-simplicial-spaces}
S_{\bullet}\sC \ra S_{\bullet}\Fun([1],\sC) \ra S_{\bullet}\sC,    
\end{equation}
where the first factor was identified with the $S_{\bullet}$-construction of the category whose objects are $X \in \sC$ with a map from $\ast \in \sC$, which is equivalent to $\sC$. So it is enough to check that (\ref{eq:fiber-sequence-simplicial-spaces}) induces a fiber sequence upon geometric realization, i.e.\ that the map
\[S_{\bullet}\Fun([1],\sC) \ra S_{\bullet}\sC\]
is a realization fibration.

This follows from Proposition \ref{prop:RF-for-F-from-rho} and Proposition \ref{prop:any-rho-is-RF} applied to the functor $\ev: \Fun(\Delta^1,\sC) \ra \sC$.
\end{proof}

\subsection{Cell decomposition theorem}
\label{subsec:weight-appendix}

In this section we give the proof of Theorem \ref{thm:cell-decomposition}. We follow closely the strategy of Waldhausen (cf. Section 1.7 in \cite{Waldhausen}). For alternate accounts the reader is refered to \cite{Fontes} and \cite{Lurie-K-theory}*{Lectures 19 and 20}.

Let $\sC$ be a Waldhausen category that admits finite colimits equipped with a bounded and non-degenerate weight structure $(\sC_{w \leq 0},\sC_{w,\geq 0})$. We introduce the category of cell decomposition.

For each $n \geq 0$ let $\sE_n(\sC)$ be defined as the following pullback
\begin{equation}
\label{eq:defn-of-cell-decomposition-categories}
    \begin{tikzcd}
    \sE_n(\sC) \ar[r] \ar[d,"q"] & \sF_n\sC \ar[d] \\
    \sC^{\heartsuit}_{w}\times \cdots \times \sC^n_{w} \ar[r] & \prod^{n}_{i=0}\sC_{w\leq i}
    \end{tikzcd}
\end{equation}
where $\sC^n_w = \sC_{w\leq n}\cap \sC_{w \geq n}$, and $\sF_n\sC$ is the category of objects $X \in \sC$ with an $n$ step filtration, i.e.\
\[\sF_n\sC = \{X_0 \hra X_1 \hra \cdots \hra X_n \; | \;  \forall i, \; X_i \in \sC\}  \] 
and the map $q$ is given by
\[q(X_0 \hra X_1 \hra \cdots \hra X_n) = (X_0,X_1/X_0,\ldots,X_n/X_{n-1}).\]

Notice that for any $n$ the category $\sC_{w \leq n}$ is a Waldhausen category, where the cofibrations are those maps that are cofibrations in $\sC$ and whose cofiber belongs to $\sC_{w \leq n}$. Similarly for $\sC_{w \geq n}$ and $\sC^n_w$. The categories $\sF_n\sC$ are also Waldhausen categories (see \cite{Barwick}*{Definition 5.6}). Thus, the category $\sE_n(\sC)$ is a Waldhausen category, since the category $\Wald$ admits finite limits (see \cite{Barwick}*{Proposition 4.4}).

We now endow the category $\sE_n(\sC)$ with a labeling, where we say that a morphism $f: X \ra Y$ belongs to $w\sE_{n}(\sC)$ if $f_n:X_n \ra Y_n$ is an equivalence.

The strategy of the proof is to apply the Fibration Theorem (Theorem \ref{thm:fibration}) to $(\hat{\sC}_n,w\hat{\sC}_n)$ and take the colimit on $n$. We first prove some lemmas that identify the categories that will show up.

\begin{lem}
\label{lem:acyclic-objects-in-cell-decomposition-category}
For $n\geq 1$, the natural map
\[\sE_n(\sC)^w \overset{\simeq}{\ra} \sC^{\heartsuit}_{w}\]
given by inclusion into $\sE_n(\sC)$ and projection into the first factor is an equivalence of Waldhausen categories.
\end{lem}

\begin{proof}
We proceed by induction on $n$. The base case is clear since it consists of objects $X_0 \hra 0$, where $X_0 \in \sC^{\heartsuit}_{w}$ and the projection just sends $X_0$ to itself in $\sC^{\heartsuit}_{w}$. 

Suppose that the result holds for $n-1$ and consider the fiber sequence
\[X_{n-1} \hra X_{n} \ra X_{n}/X_{n-1},\]
where $X_n \simeq 0$. By definition of $\sE_n(\sC)$ we have that $X_n/X_{n-1} \in \sC^{n}_w$, hence one has
\[
\tau_{w\leq(n-1)}(X_n) \simeq X_{n-1}
\]
since $X_{n-1}\in \sC_{w\leq (n-1)}$, this gives that $X_{n-1}\simeq 0$.
\end{proof}

\begin{lem}
\label{lem:cell-decomposition-category}
For $n\geq 1$, one has an equivalence of Waldhausen categories
\[\alpha_n: \sE_n(\sC)^w\times\sC^{1}_w \times \cdots\times \sC^n_w \overset{\simeq}{\ra} \sE_n(\sC).\]
\end{lem}

\begin{proof}
Again we proceed by induction. For $n=1$ one has the functor
\[\alpha_1((X_0 \hra 0),Y) = (X_0 \hra X_0\oplus Y)\]
which clearly gives an equivalence.

For general $n$ we consider the functor $\beta: \sE_{n-1}(\sC)\times \sC^n_{w} \ra \sE_n(\sC)$ given by
\[\beta((X_0 \hra \cdots X_{n-1}),X_{n}) = (X_0,\ldots, X_{n-1},X_{n-1}\oplus X_n).\]
This also gives an equivalence so the result for $n$ follows from the result for $n-1$.
\end{proof}

\begin{lem}
\label{lem:colimit:cell-decomposition-category}
One has an equivalence of Waldhausen categories
\[\colim_{n \geq 1} \sE_n(\sC) \overset{\simeq}{\ra} \sC.\]
\end{lem}

\begin{proof}
We first note that an object $X \in \colim_{n \geq 1}\sE_n(\sC)$ is a filtered object, such that
\[X_n \simeq X_{n+1}\]
for some $n >> 0$. Thus, the map from the colimit is just sending an object to its stable part. To check that this is an equivalence we use the bounded condition of the weight structure on $\sC$. Thus, the map $X \mapsto (\tau_{w \leq 0}X \hra \tau_{w \leq 1}X \hra \ldots \hra \tau_{w \leq n}X \simeq \tau_{w \leq n+1}X \hra \ldots)$ is an inverse to the previous maps.
\end{proof}

\begin{lem}
\label{lem:cell-decomposition-category-has-enough-cofibrations}
For every $n$ the labeled $\infty$-Waldhausen category $(\sE_n(\sC),w\sE_n(\sC))$ has enough cofibrations. 
\end{lem}

\begin{proof}
By Definition \ref{defn:enough-cofibrations} we need to produce an endofunctor $F$ of $\Fun(\Delta^1,\sE_n(\sC))$ and a natural transformation $\eta: \id \ra F$, satisfying the three properties of the definition.

Given an object $(X \overset{f}{\ra} Y) \in \Fun(\Delta^1,\sE_n(\sC)$ represented by a diagram
\begin{equation}
\label{eq:morphism-of-n-filtered-objects}
\begin{tikzcd}
X_0 \ar[r,hook] \ar[d,"f_0"] & X_1 \ar[r,hook] \ar[d,"f_1"] & \cdots \ar[r,hook] & X_{n-1} \ar[r,hook] \ar[d,"f_{n-1}"] & X_n \ar[d,"f_n"] \\
Y_0 \ar[r,hook] & Y_1 \ar[r,hook] & \cdots \ar[r,hook] & Y_{n-1} \ar[r,hook] & Y_n
\end{tikzcd}    
\end{equation}
we inductively define $F(X \overset{f}{\ra} Y)$ as follows
\[
F(X_n) = X_n \; \mbox{ and }\; F(Y_n) = Y_n,
\]
and for $i < n$, we let
\[
F(X_i) = Y_i\times_{Y_{i+1}}F(X_{i+1}).
\]

The natural transformation $\eta: \id \Rightarrow F$ is the natural map from $X_i$ to the pullback $Y_{i}\times_{Y_n}X_n \simeq F(X_i)$. Notice that the maps $F(X_i) \hra F(X_{i+1})$ are cofibrations because those are isomorphic to the maps $Y_i \hra Y_{i+1}$.

We now check the conditions of Definition \ref{defn:enough-cofibrations}.

Condition (i): Recall that $f$ is labeled if $f_{n}$ is an equivalence, thus we need to argue that if $f_n$ is an equivalence then (\ref{eq:morphism-of-n-filtered-objects}) is a cofibration in $\sF_n(\sC)$. Barwick describes the cofibrations in $\sF_n(\sC)$ as diagrams (\ref{eq:morphism-of-n-filtered-objects}) where any square is a cofibration as a morphism in $\sF_1(\sC)$. Then Lemma 5.8. from \cite{Barwick} characterizes a map in $\sF_1(\sC)$ as a cofibration if $g_0$ is a cofibration in the diagram
\[
\begin{tikzcd}
Z_0 \ar[r,hook] \ar[d,"g_0"] & Z_1 \ar[d,"g_1"] \\
W_0 \ar[r,hook] & W_1,
\end{tikzcd}
\]
and for any $U$ extending the above diagram to
\[
\begin{tikzcd}
Z_0 \ar[r,hook] \ar[d,"g_0"] & Z_1 \ar[d,"g_1"] \ar[rdd] & \\
W_0 \ar[r,hook] \ar[rrd] & U \ar[rd,"h"] \arrow[lu, phantom, "\lrcorner", very near start] & \\
& & W_1
\end{tikzcd}
\]
the map $h$ is a cofibration.

In our example we obtain the diagram
\[
\begin{tikzcd}
Y_{n-1}\times_{Y_{n}}X_n\ar[r,hook] \ar[d,"\simeq"] & X_n \ar[d,"f'_n"] \ar[rdd,"f_n"] & \\
Y_{n-1} \ar[r,hook] \ar[rrd] & U \ar[rd,"h"] & \\
& & Y_n
\end{tikzcd}
\]
where the map $f'_n$ is an equivalence, since it is the pushout of an equivalence, and $f_n$ is an equivalence by assumption. Thus, $h$ is an equivalence, in particular it is also a cofibration.

Condition (ii): We need to show that $\eta_f$ is an equivalence for any $f$ a labeled cofibration. We recall that Barwick defined the cofibrations on $\sF_1(\sC)$ as the smallest class generated by 
\[
\mbox{a) }
\begin{tikzcd}
Z_0 \ar[r,hook] \ar[d,"\simeq"] & Z_1 \ar[d,hook] \\
W_0 \ar[r,hook] & W_1
\end{tikzcd}
\;\;\mbox{ and }\;\;
\mbox{b) }
\begin{tikzcd}
Z_0 \ar[r,hook] \ar[d,hook] & Z_1 \ar[d,hook] \\
W_0 \ar[r,hook] & W_1 \arrow[lu, phantom, "\lrcorner", very near start]
\end{tikzcd}
\]

We notice that by construction of $F$ we only need to check that
\[X_{i} \overset{\simeq}{\ra} F(X_i) = Y_i\times_{Y_{i+1}}F(X_{i+1})\]
for $0 \leq i \leq n-1$. Since $F(X_i) \simeq Y_i$ it is enough to check that 
\[f_i: X_i \ra Y_i\]
is an equivalence for all $0\leq i\leq n-1$ when $f_n$ is an equivalence and a cofibration.

By applying the description of cofibrations we see that either $f_n$ is the pushout of $f_{n-1}$ by b), or that $f_{n-1}$ is an isomorphism. In either case we obtain by induction that $f_i$ is an isomorphism for all $0 \leq i \leq n$.

Condition (iii): We need to check that if $f$ is labeled, then $\eta_f$ is objectwise labeled, i.e.\ $X_n \ra F(X_n)$ is an equivalence, which follows from the definiton of $F$.
\end{proof}

\begin{thm}
\label{thm:cell-decomposition-appendix}
Let $\sC$ be a Waldhausen category that admits finite colimits equipped with a bounded and non-degenerate weight structure $(\sC_{w \leq 0},\sC_{w,\geq 0})$. Then the canonical inclusion map
\[
\K(\sC^{\heartsuit}_{w}) \ra \K(\sC)
\]
is an equivalence of K-theory spectra.
\end{thm}

\begin{proof}
Consider $(\sE_n(\sC),w\sE_n(\sC)$, which is a labeled Waldhausen category with enough fibrations by Lemma \ref{lem:cell-decomposition-category-has-enough-cofibrations}, so we can apply Theorem \ref{thm:fibration} to obtain the following fiber sequence
\[
\begin{tikzcd}
\K(\sE_n(\sC)^w) \ar[r] \ar[d] & \K(\sE_n(\sC)) \ar[d] \\
\ast \ar[r] & \jmath_!\K(\sB(\sE_n(\sC),w\sE_n(\sC))).
\end{tikzcd}
\]
By Lemma \ref{lem:acyclic-objects-in-cell-decomposition-category} and Lemma \ref{lem:cell-decomposition-category} we can rewrite the upper row of the above diagram to get
\begin{equation}
\label{eq:fibration-simplified-n-cell-category}
\begin{tikzcd}
\K(\sC^{\heartsuit}_{w}) \ar[r] \ar[d] & \K(\sC^{\heartsuit}_w\times\sC^{1}_w \times \cdots\times \sC^n_w) \ar[d] \\
\ast \ar[r] & \jmath_!\K(\sB(\sE_n(\sC),w\sE_n(\sC))).
\end{tikzcd}
\end{equation}
Thus, by the Additivity Theorem (see Theorem \ref{thm:additivity}), we identify
\begin{equation}
\label{eq:identification-of-n-th-piece}
\jmath_!\K(\sB(\sE_n(\sC),w\sE_n(\sC))) \simeq \K(\sC^1_w\times\cdots\times\sC^n_w).    
\end{equation}

Finally, by taking the colimit on $n$ of the diagram (\ref{eq:fibration-simplified-n-cell-category}) and using that filtered colimits commute with fiber sequences we obtain
\[
\begin{tikzcd}
\K(\colim_n \sC^{\heartsuit}_w) \ar[r] \ar[d] & \K(\colim_{n\geq 1}\sE_{n}(\sC)) \ar[d] \\
\ast \ar[r] & \colim_{n\geq 1}\K(\prod^n_{i=1}\sC^{i}_w)
\end{tikzcd}
\]
The term on the lower right corner vanishes by an Eilenberg swindle argument and Lemma \ref{lem:colimit:cell-decomposition-category} gives the desired result.

This finishes the proof of the theorem.
\end{proof}

\begin{rem}
Consider the following modification of diagram (\ref{eq:defn-of-cell-decomposition-categories})
\[
\begin{tikzcd}
    \sE_n(\sC) \ar[r] \ar[d,"t"] & \sF_n\sC \ar[d] \\
    \sC_{w\leq 0} \times \cdots \times \sC_{w\leq (n-1)} \times \sC_{w\leq n} \ar[r] & \prod^{n}_{i=0}\sC_{w\leq i}
    \end{tikzcd}
\]
where $t$ is given by
\[t(X_0 \hra X_1 \hra \cdots \hra X_n) = (\tau_{w\leq 0}X_1,\tau_{w\leq 1}X_2,\ldots,\tau_{w\leq (n-1)}X_n, X_n).\]
We can proceed exactly as we did in this section to obtain the result that the canonical map
\[
\K(\sC_{w\leq 0}) \overset{\simeq}{\ra} \K(\sC)
\]
is an equivalence of spectra, for $\sC$ satisfying the same conditions as in Theorem \ref{thm:cell-decomposition-appendix}.
\end{rem}

\subsection{Additive K-theory}
\label{subsec:additive-K-theory-appendix}

In this section we prove Theorem \ref{thm:additive-K-theory-agrees-when-split-cofibrations}. 

Recall that we have $\sC$, a Waldhausen category which admits finite direct sums. We constructed a map in Remark \ref{rem:map-additive-to-K-theory} which induces a map between the underlying $\infty$-groupoids:
\begin{equation}
\label{eq:comparison-add-Waldhausen-K-theory}
G_n: \B_n\sC \ra S_n\sC    
\end{equation}
given by $G_n(X_1,\ldots,X_n) = (X_1 \hra X_1\oplus X_2 \hra \cdots \hra \oplus^n_{i=1}X_i)$.

\begin{rem}
\label{rem:splitting-in-additive-categories}
For a pointed category $\sC$, if its homotopy category $\sh\sC$ is additive, then for every $X,Y \in \sC$ the canonical map
\[X \sqcup Y \ra X\times Y\]
admits a splitting.
\end{rem}

The following lemma is key prove the comparison result.

\begin{lem}
\label{lem:bar-construction-contractible}
Suppose that $\sC$ is a Waldhausen category with split cofibrations and finite direct sums such that $\h\sC$ is additive. For every $n\geq 2$, consider the action of $\B_{n-1}\sC$ on $S_{n}\sC$ given as
\begin{align*}
    a_n: \B_{n-1}\sC\times S_n\sC & \ra S_n\sC \\
    ((Y_1,\ldots,Y_{n-1},(X_1\ra \cdots \ra X_n)) & \mapsto (X_1 \ra X_1\oplus Y_1 \ra X_1\oplus Y_1 \oplus Y_2 \ra \cdots \ra X_n\oplus \oplus^{n-1}_{i=1}Y_i)
\end{align*}

Then 
\[S_{n}\sC\otimes_{\B_{n-1}\sC}\ast \ra \sC^{\simeq},\] 
i.e. the one-sided bar construction, is an equivalence in the category of spaces.
\end{lem}

\begin{proof}
It is enough to consider the fiber of the above map over a fixed object $Z \in \sC^{\simeq}$. Let $S_{n}\sC_{Z/}$ be the subcategory of $S_n\sC$ spanned by objects of the form
\[(Z \ra X_{2} \ra \cdots \ra X_n).\]
This subcategory has coproducts given by taking the coproduct relative to $Z$, thus its underlying $\infty$-groupoid $S_n\sC_{Z/}$ is an $\Einf$-monoid in spaces. The action $a_n$ on $S_n\sC$ restricts to the sub-$\infty$-groupoid $S_n\sC_{Z/}$ and we claim that 
\begin{equation}
\label{eq:one-sided-bar-construction-on-a-point}
S_n\sC_{Z/}\otimes_{\B_{n-1}\sC}\ast \simeq \ast.    
\end{equation}

We have a map
\[g_{n-1}: \B_{n-1}\sC \ra S_n\sC_{Z/}\]
given by $g_{n-1}((Y_i)_I) = a_{n}((Y_i)_I,(Z \ra Z \ra \cdots \ra Z))$ and equation (\ref{eq:one-sided-bar-construction-on-a-point}) is equivalent to 
\begin{equation}
\label{eq:cofib-g-contractible}
    \Cofib(g_{n-1}) \simeq \ast.
\end{equation}
We notice that $g_{n-1}$ is surjective on $\pi_0$, thus $\Cofib(g_{n-1})$ is automatically grouplike, hence it is contractible if and only if its group completion is contractible.

Now we consider the map
\begin{align*}
    S_n\sC_{Z/} & \overset{q_{n-1}}{\ra} \B_{n-1}\sC \\
    (Z \ra X_2 \ra X_3 \ra \cdots \ra X_n) & \mapsto (X_2/Z,X_3/X_2,\cdots,X_n/X_{n-1}).
\end{align*}

From the definition it is clear that $q_{n-1}\circ g_{n-1}$ is the identity on $\B_{n-1}\sC$. We check that the opposite composition naturally satisfies
\begin{equation}
\label{eq:disjoint-identity}
(g_{n-1}\circ q_{n-1}) \sqcup \id_{S_n\sC_{Z/}} \simeq \id_{S_n\sC_{Z/}} \sqcup \id_{S_n\sC_{Z/}},    
\end{equation}
which will prove that $(g_{n-1}\circ q_{n-1})$ is homotopic to the identity upon group completion.

Indeed, the left-hand side of (\ref{eq:disjoint-identity}) applied to $(Z\ra X_2 \ra \cdots \ra X_{n})$ gives
\[Z \ra X_2/Z \oplus Z \oplus Z_2/Z \ra X_3/X_2\oplus X_2/Z \oplus X \oplus X_3/X_2 \oplus X_2/Z \ra \cdots \]
whereas the right-hand side of (\ref{eq:disjoint-identity}) applied to $(Z\ra X_2 \ra \cdots \ra X_{n})$ gives
\[Z \ra X_2\sqcup_{Z}X_2 \ra X_3\sqcup_{Z}X_3 \ra \cdots.\]

However, Remark \ref{rem:splitting-in-additive-categories} implies that the fold map
\[X_2\sqcup_{Z}X_2 \ra X_2\]
has a canonical splitting by the inclusion of $X_2$, i.e.
\[X_2\sqcup_{Z}X_2 \simeq X_2 \oplus (X_2/Z).\]

This finishes the proof of the lemma.
\end{proof}

\begin{prop}
\label{prop:Kadd-is-K-for-split}
Let $\sC$ be a Waldhausen category with split cofibrations and finite direct sums. Moreover assume that $\sh\sC$ is an additive category\footnote{For $\sC$ a prestable category this condition is automatically satisfied, see \cite{HA}*{Lemma. 1.1.2.9}.}, then the maps (\ref{eq:comparison-add-Waldhausen-K-theory}) induce an equivalence
\[\Ka(\sC) \ra \K(\sC).\]
\end{prop}

\begin{proof}
Note that due to the Dold-Kan equivalence, we can think of the spectra $\Ka(\sC)$ and $\K(\sC)$ as grouplike $\Einf$-monoids. That is, one has an adjuction
\begin{equation}
    \label{eq:Eoo-monoid-spaces-adjunction}
    \begin{tikzcd}
    \EinfMon \ar[r,swap,"(-)^{\rm gp}", shift right = 1] & \ar[l,swap,"\Omega^{\infty}",shift right = 1] \Spctr
    \end{tikzcd}
\end{equation}
between $\Einf$-monoids and spectra.

Thus, applying the unit of (\ref{eq:Eoo-monoid-spaces-adjunction}) to $\Sigma\Ka(\sC)$ we obtain
\[\Sigma\Ka(\sC) = \left|\B_{\bullet}\sC\right| \overset{\simeq}{\ra} \Omega^{\infty}\left|\B_{\bullet}\sC\right|^{\rm gp} \simeq \Omega^{\infty}\left|\B_{\bullet}\sC^{\rm gp}\right|,\]
where in the last term we apply the group completion level-wise and the equivalence follows from the fact that group completion commutes with colimits.

Similarly, one has
\[\Sigma\K(\sC) = \left|S_{\bullet}\sC\right| \overset{\simeq}{\ra} \Omega^{\infty}\left|S_{\bullet}\sC\right|^{\rm gp} \simeq \Omega^{\infty}\left|S_{\bullet}\sC^{\rm gp}\right|,\]
where as above we apply group completion level-wise and the last equivalence follows for the same reason.

Hence its is enough to show that the maps (\ref{eq:comparison-add-Waldhausen-K-theory}) induce equivalences after group completion.

For each $n\geq 2$ we consider the diagram
\begin{equation}
    \begin{tikzcd}
    \B_{n-1}\sC \ar[r,"\imath"] \ar[d] & \B_n\sC \ar{r}{f_n} \ar[d,"\pi_1"] & S_n\sC \ar[d,"\ev_1"] \\
    \ast \ar[r] & \sC^{\simeq} \ar{r}{\id} & \sC^{\simeq}
    \end{tikzcd}
\end{equation}
where $\imath'_1(X_1,\ldots,X_{n-1}) = (0,X_1,\ldots,X_{n-1})$, $\pi_1$ is the projection onto the first factor and $\ev_1(X_1 \ra \cdots \ra X_n) = X_1$.

We want to prove that the maps $f_{n}$ are equivalences after group completion. It is enough to prove that the right squares are pullback squares after group completion. Since the category $\EinfMon$ is stable, this is the same as proving that the right square is a pushout square after group completion. We notice that the left square is a pushout diagram, so we are reduced to proving that the outer square is a pushout diagram after group completion. This follows from Lemma \ref{lem:bar-construction-contractible}.
\end{proof}








\begin{bibdiv}
\begin{biblist}

\bib{AGH}{article}{
      author={Antieau, Benjamin},
      author={Gepner, David},
      author={Heller, Jeremiah},
       title={K-theoretic obstructions to bounded t-structures},
        date={2019},
        ISSN={0020-9910},
     journal={Invent. Math.},
      volume={216},
      number={1},
       pages={241\ndash 300},
         url={https://doi.org/10.1007/s00222-018-00847-0},
      review={\MR{3935042}},
}

\bib{Ausoni-Rognes}{article}{
      author={Ausoni, Christian},
      author={Rognes, John},
       title={Rational algebraic {$K$}-theory of topological {$K$}-theory},
        date={2012},
        ISSN={1465-3060},
     journal={Geom. Topol.},
      volume={16},
      number={4},
       pages={2037\ndash 2065},
  url={https://doi-org.proxy2.library.illinois.edu/10.2140/gt.2012.16.2037},
      review={\MR{2975299}},
}

\bib{Barwick-heart}{article}{
      author={Barwick, Clark},
       title={On exact {$\infty$}-categories and the theorem of the heart},
        date={2015},
        ISSN={0010-437X},
     journal={Compos. Math.},
      volume={151},
      number={11},
       pages={2160\ndash 2186},
         url={https://doi.org/10.1112/S0010437X15007447},
      review={\MR{3427577}},
}

\bib{Barwick}{article}{
      author={Barwick, Clark},
       title={On the algebraic {$K$}-theory of higher categories},
        date={2016},
        ISSN={1753-8416},
     journal={J. Topol.},
      volume={9},
      number={1},
       pages={245\ndash 347},
         url={https://doi.org/10.1112/jtopol/jtv042},
      review={\MR{3465850}},
}

\bib{Bhatt-Scholze}{article}{
      author={Bhatt, Bhargav},
      author={Scholze, Peter},
       title={Projectivity of the {W}itt vector affine {G}rassmannian},
        date={2017},
        ISSN={0020-9910},
     journal={Invent. Math.},
      volume={209},
      number={2},
       pages={329\ndash 423},
         url={https://doi.org/10.1007/s00222-016-0710-4},
      review={\MR{3674218}},
}

\bib{BGT}{article}{
      author={Blumberg, Andrew~J.},
      author={Gepner, David},
      author={Tabuada, Gon\c{c}alo},
       title={A universal characterization of higher algebraic $k$-theory},
        date={2013},
        ISSN={1465-3060},
     journal={Geom. Topol.},
      volume={17},
      number={2},
       pages={733\ndash 838},
         url={https://doi.org/10.2140/gt.2013.17.733},
      review={\MR{3070515}},
}

\bib{Bondarko}{article}{
      author={Bondarko, M.~V.},
       title={Weight structures vs. {$t$}-structures; weight filtrations,
  spectral sequences, and complexes (for motives and in general)},
        date={2010},
        ISSN={1865-2433},
     journal={J. K-Theory},
      volume={6},
      number={3},
       pages={387\ndash 504},
         url={https://doi.org/10.1017/is010012005jkt083},
      review={\MR{2746283}},
}

\bib{Bondarko-summary}{article}{
      author={{Bondarko}, Mikhail~V.},
       title={{Weight structures and motives; comotives, coniveau and
  Chow-weight spectral sequences, and mixed complexes of sheaves: a survey}},
        date={2009Mar},
     journal={arXiv e-prints},
       pages={arXiv:0903.0091},
      eprint={0903.0091},
}

\bib{BGW-Tate}{article}{
      author={Braunling, Oliver},
      author={Groechenig, Michael},
      author={Wolfson, Jesse},
       title={Tate objects in exact categories},
        date={2016},
        ISSN={1609-3321},
     journal={Mosc. Math. J.},
      volume={16},
      number={3},
       pages={433\ndash 504},
        note={With an appendix by Jan \v{S}\v{t}ov\'{i}\v{c}ek and Jan
  Trlifaj},
      review={\MR{3510209}},
}

\bib{BGW-Index}{article}{
      author={Braunling, Oliver},
      author={Groechenig, Michael},
      author={Wolfson, Jesse},
       title={The index map in algebraic {$K$}-theory},
        date={2018},
        ISSN={1022-1824},
     journal={Selecta Math. (N.S.)},
      volume={24},
      number={2},
       pages={1039\ndash 1091},
         url={https://doi.org/10.1007/s00029-017-0364-0},
      review={\MR{3782417}},
}

\bib{EHKSY}{article}{
      author={{Elmanto}, Elden},
      author={{Hoyois}, Marc},
      author={{Khan}, Adeel~A.},
      author={{Sosnilo}, Vladimir},
      author={{Yakerson}, Maria},
       title={{Modules over algebraic cobordism}},
        date={2019-08},
     journal={arXiv e-prints},
       pages={arXiv:1908.02162},
      eprint={1908.02162},
}

\bib{EKMM}{book}{
      author={Elmendorf, Anthony~D},
       title={Rings, modules, and algebras in stable homotopy theory},
   publisher={American Mathematical Soc.},
        date={1997},
      number={47},
}

\bib{Faltings}{article}{
      author={Faltings, Gerd},
       title={Algebraic loop groups and moduli spaces of bundles},
        date={2003},
        ISSN={1435-9855},
     journal={J. Eur. Math. Soc. (JEMS)},
      volume={5},
      number={1},
       pages={41\ndash 68},
         url={https://doi.org/10.1007/s10097-002-0045-x},
      review={\MR{1961134}},
}

\bib{FHK}{article}{
      author={Faonte, Giovanni},
      author={Hennion, Benjamin},
      author={Kapranov, Mikhail},
       title={Higher {K}ac-{M}oody algebras and moduli spaces of
  {$G$}-bundles},
        date={2019},
        ISSN={0001-8708},
     journal={Adv. Math.},
      volume={346},
       pages={389\ndash 466},
         url={https://doi.org/10.1016/j.aim.2019.01.040},
      review={\MR{3910800}},
}

\bib{Fontes}{article}{
      author={{Fontes}, Ernest~E.},
       title={{Weight structures and the algebraic $K$-theory of stable
  $\infty$-categories}},
        date={2018Dec},
     journal={arXiv e-prints},
       pages={arXiv:1812.09751},
      eprint={1812.09751},
}

\bib{FZ}{article}{
      author={Frenkel, Edward},
      author={Zhu, Xinwen},
       title={Gerbal representations of double loop groups},
        date={2012},
        ISSN={1073-7928},
     journal={Int. Math. Res. Not. IMRN},
      number={17},
       pages={3929\ndash 4013},
         url={https://doi.org/10.1093/imrn/rnr159},
      review={\MR{2972546}},
}

\bib{GR-I}{book}{
      author={Gaitsgory, Dennis},
      author={Rozenblyum, Nick},
       title={A study in derived algebraic geometry. {V}ol. {I}.
  {C}orrespondences and duality},
      series={Mathematical Surveys and Monographs},
   publisher={American Mathematical Society, Providence, RI},
        date={2017},
      volume={221},
        ISBN={978-1-4704-3569-1},
      review={\MR{3701352}},
}

\bib{GR-II}{book}{
      author={Gaitsgory, Dennis},
      author={Rozenblyum, Nick},
       title={A study in derived algebraic geometry. {V}ol. {II}.
  {D}eformations, {L}ie theory and formal geometry},
      series={Mathematical Surveys and Monographs},
   publisher={American Mathematical Society, Providence, RI},
        date={2017},
      volume={221},
        ISBN={978-1-4704-3570-7},
      review={\MR{3701353}},
}

\bib{Gruson}{incollection}{
      author={Gruson, Laurent},
       title={Une propri\'{e}t\'{e} des couples hens\'{e}liens},
        date={1972},
   booktitle={Colloque d'{A}lg\`ebre {C}ommutative ({R}ennes, 1972), {E}xp.
  {N}o. 10},
       pages={13},
      review={\MR{0412187}},
}

\bib{Hennion-Tate}{article}{
      author={Hennion, Benjamin},
       title={Tate objects in stable {$(\infty,1)$}-categories},
        date={2017},
        ISSN={1532-0073},
     journal={Homology Homotopy Appl.},
      volume={19},
      number={2},
       pages={373\ndash 395},
         url={https://doi.org/10.4310/HHA.2017.v19.n2.a18},
      review={\MR{3731483}},
}

\bib{Sketches-I}{book}{
      author={Johnstone, Peter~T.},
       title={Sketches of an elephant: a topos theory compendium. {V}ol. 1},
      series={Oxford Logic Guides},
   publisher={The Clarendon Press, Oxford University Press, New York},
        date={2002},
      volume={43},
        ISBN={0-19-853425-6},
      review={\MR{1953060}},
}

\bib{Knudsen-Mumford}{article}{
      author={Knudsen, Finn},
      author={Mumford, David},
       title={The projectivity of the moduli space of stable curves. {I}:
  {Preliminaries} on "det" and "{Div}".},
    language={en},
        date={1976-06},
        ISSN={1903-1807},
     journal={MATHEMATICA SCANDINAVICA},
      volume={39},
       pages={19\ndash 55},
         url={https://www.mscand.dk/article/view/11642},
}

\bib{HTT}{book}{
      author={Lurie, Jacob},
       title={Higher topos theory},
      series={Annals of Mathematics Studies},
   publisher={Princeton University Press, Princeton, NJ},
        date={2009},
      volume={170},
        ISBN={978-0-691-14049-0; 0-691-14049-9},
         url={https://doi.org/10.1515/9781400830558},
      review={\MR{2522659}},
}

\bib{HA}{article}{
      author={Lurie, Jacob},
       title={Higher algebra},
        date={2012},
     journal={Preprint, available at
  \url{http://www.math.harvard.edu/~lurie/papers/higheralgebra.pdf}},
}

\bib{Lurie-K-theory}{unpublished}{
      author={Lurie, Jacob},
       title={Algebraic $k$-theory and manifold topology},
        date={2014},
        note={Lecture notes, available at
  \url{http://www.math.harvard.edu/~lurie/281.html}},
}

\bib{SAG}{unpublished}{
      author={Lurie, Jacob},
       title={Spectral algebraic geometry},
        date={2018},
}

\bib{McCarthy}{article}{
      author={McCarthy, Randy},
       title={On fundamental theorems of algebraic {$K$}-theory},
        date={1993},
        ISSN={0040-9383},
     journal={Topology},
      volume={32},
      number={2},
       pages={325\ndash 328},
         url={https://doi.org/10.1016/0040-9383(93)90023-O},
      review={\MR{1217072}},
}

\bib{Osipov-Zhu}{article}{
      author={Osipov, Denis},
      author={Zhu, Xinwen},
       title={The two-dimensional {C}ontou-{C}arr\`ere symbol and reciprocity
  laws},
        date={2016},
        ISSN={1056-3911},
     journal={J. Algebraic Geom.},
      volume={25},
      number={4},
       pages={703\ndash 774},
         url={https://doi-org.proxy2.library.illinois.edu/10.1090/jag/664},
      review={\MR{3533184}},
}

\bib{Rezk}{unpublished}{
      author={Rezk, Charles},
       title={When are homotopy colimits compatible with homotopy base
  change?},
        date={2014},
        note={Notes available at
  \url{https://faculty.math.illinois.edu/~rezk/i-hate-the-pi-star-kan-condition.pdf}},
}

\bib{Rotman}{book}{
      author={Rotman, Joseph~J.},
       title={An introduction to homological algebra},
     edition={Second},
      series={Universitext},
   publisher={Springer, New York},
        date={2009},
        ISBN={978-0-387-24527-0},
         url={https://doi.org/10.1007/b98977},
      review={\MR{2455920}},
}

\bib{STV}{article}{
      author={Sch\"{u}rg, Timo},
      author={To\"{e}n, Bertrand},
      author={Vezzosi, Gabriele},
       title={Derived algebraic geometry, determinants of perfect complexes,
  and applications to obstruction theories for maps and complexes},
        date={2015},
        ISSN={0075-4102},
     journal={J. Reine Angew. Math.},
      volume={702},
       pages={1\ndash 40},
         url={https://doi.org/10.1515/crelle-2013-0037},
      review={\MR{3341464}},
}

\bib{Sosnilo}{article}{
      author={{Sosnilo}, Vladimir},
       title={{Theorem of the heart in negative K-theory for weight
  structures}},
        date={2017May},
     journal={arXiv e-prints},
       pages={arXiv:1705.07995},
      eprint={1705.07995},
}

\bib{stacks-project}{misc}{
      author={{Stacks Project Authors}, The},
       title={\textit{Stacks Project}},
         how={\url{https://stacks.math.columbia.edu}},
        date={2018},
}

\bib{Thomason-etale}{article}{
      author={Thomason, R.~W.},
       title={Algebraic {$K$}-theory and \'{e}tale cohomology},
        date={1985},
        ISSN={0012-9593},
     journal={Ann. Sci. \'{E}cole Norm. Sup. (4)},
      volume={18},
      number={3},
       pages={437\ndash 552},
         url={http://www.numdam.org/item?id=ASENS_1985_4_18_3_437_0},
      review={\MR{826102}},
}

\bib{TT}{incollection}{
      author={Thomason, R.~W.},
      author={Trobaugh, Thomas},
       title={Higher algebraic {$K$}-theory of schemes and of derived
  categories},
        date={1990},
   booktitle={The {G}rothendieck {F}estschrift, {V}ol. {III}},
      series={Progr. Math.},
      volume={88},
   publisher={Birkh\"{a}user Boston, Boston, MA},
       pages={247\ndash 435},
         url={https://doi.org/10.1007/978-0-8176-4576-2_10},
      review={\MR{1106918}},
}

\bib{Tits}{incollection}{
      author={Tits, Jacques},
       title={Groups and group functors attached to {K}ac-{M}oody data},
        date={1985},
   booktitle={Workshop {B}onn 1984 ({B}onn, 1984)},
      series={Lecture Notes in Math.},
      volume={1111},
   publisher={Springer, Berlin},
       pages={193\ndash 223},
         url={https://doi.org/10.1007/BFb0084591},
      review={\MR{797422}},
}

\bib{Toen-descent-n-stacks}{article}{
      author={To\"{e}n, Bertrand},
       title={Descente fid\`element plate pour les {$n$}-champs d'{A}rtin},
        date={2011},
        ISSN={0010-437X},
     journal={Compos. Math.},
      volume={147},
      number={5},
       pages={1382\ndash 1412},
         url={https://doi.org/10.1112/S0010437X10005245},
      review={\MR{2834725}},
}

\bib{HAG-II}{article}{
      author={To\"{e}n, Bertrand},
      author={Vezzosi, Gabriele},
       title={Homotopical algebraic geometry. {II}. {G}eometric stacks and
  applications},
        date={2008},
        ISSN={0065-9266},
     journal={Mem. Amer. Math. Soc.},
      volume={193},
      number={902},
       pages={x+224},
         url={https://doi-org.proxy2.library.illinois.edu/10.1090/memo/0902},
      review={\MR{2394633}},
}

\bib{Waldhausen}{incollection}{
      author={Waldhausen, Friedhelm},
       title={Algebraic {$K$}-theory of spaces},
        date={1985},
   booktitle={Algebraic and geometric topology ({N}ew {B}runswick, {N}.{J}.,
  1983)},
      series={Lecture Notes in Math.},
      volume={1126},
   publisher={Springer, Berlin},
       pages={318\ndash 419},
         url={https://doi.org/10.1007/BFb0074449},
      review={\MR{802796}},
}

\end{biblist}
\end{bibdiv}
\end{document}